\documentclass[a4paper,12pt,reqno]{amsart}
\usepackage[T1]{fontenc}
\usepackage[utf8]{inputenc}
\usepackage[libertinus,vvarbb]{newtx}
\usepackage[margin=1in]{geometry}
\usepackage{enumitem}
\usepackage{graphicx}
\usepackage[svgnames]{xcolor}
\usepackage{xparse}
\usepackage{xurl}
\usepackage[bookmarksdepth=2]{hyperref}

\hypersetup{colorlinks=true,urlcolor=MidnightBlue,linkcolor=MidnightBlue,citecolor=MidnightBlue}

\usepackage[normalem]{ulem}

\newcommand{\stkout}[1]{\ifmmode\text{\sout{\ensuremath{#1}}}\else\sout{#1}\fi}
\newcommand{\crsout}[1]{\ifmmode\text{\xout{\ensuremath{#1}}}\else\xout{#1}\fi}

\numberwithin{equation}{section}

\newtheorem{theorem}{Theorem}[section]
\newtheorem{proposition}[theorem]{Proposition}
\newtheorem{lemma}[theorem]{Lemma}
\newtheorem{corollary}[theorem]{Corollary}

\theoremstyle{remark}
\newtheorem*{notation}{Notation}
\newtheorem*{structure}{Structure of the paper}

\DeclareMathOperator{\re}{Re}
\DeclareMathOperator{\im}{Im}

\DeclareMathOperator{\dist}{dist}

\newcommand{\ind}{\mathbb{1}}
\newcommand{\C}{\mathbb{C}}
\newcommand{\R}{\mathbb{R}}

\newcommand{\sph}{\mathbb{S}}
\newcommand{\ball}{\mathbb{B}}
\newcommand{\fourier}{\mathscr{F}}
\newcommand{\eps}{\varepsilon}
\newcommand{\ph}{\varphi}
\newcommand{\thet}{\vartheta}

\renewcommand{\le}{\leqslant}
\renewcommand{\ge}{\geqslant}

\newcommand{\lv}{\lvert}
\newcommand{\rv}{\rvert}
\newcommand{\lV}{\lVert}
\newcommand{\rV}{\rVert}

\newcommand{\F}[4]{{_2F_1}\biggl( \genfrac{}{}{0pt}{0}{{#1}, \;\; {#2}}{{#3}} \; \bigg| \; {#4} \biggr)}

\NewDocumentCommand{\formula}{ssom}{%
 \IfBooleanTF{#1}{%
  \IfBooleanTF{#2}{%
   \IfValueTF{#3}%
    {\begin{align}\label{#3}\begin{gathered}#4\end{gathered}\end{align}}%
    {\begin{gather}#4\end{gather}}%
  }{%
   \IfValueTF{#3}%
    {\begin{align}\label{#3}\begin{aligned}#4\end{aligned}\end{align}}%
    {\begin{gather*}#4\end{gather*}}%
  }%
 }{%
  \IfValueTF{#3}%
   {\begin{align}\label{#3}#4\end{align}}%
   {\begin{align*}#4\end{align*}}%
 }%
}

\newcommand{\nist}[2]{\href{https://dlmf.nist.gov/#1.E#2}{Eq.~#1.#2}}
\newcommand{\doi}[1]{\href{https://doi.org/#1}{\textsf{\scriptsize DOI:#1}}}
\newcommand{\arxiv}[1]{\href{https://arxiv.org/abs/#1}{\textsf{\scriptsize arXiv:#1}}}
\newcommand{\isbn}[1]{\textsf{ISBN:#1}}
\NewDocumentCommand{\link}{oom}{\href{#3}{\textsf{\scriptsize \IfValueTF{#1}{#1}{#3}}}\IfValueT{#2}{{\textsf{\scriptsize #2}}}}

\begin{document}

\title[Shorter proof of dimension-free $L^p$ estimates for maximal Riesz transforms]{Shorter proof of dimension-free $L^p$ estimates \\ for maximal Riesz transforms}
\author{Maciej Kucharski, Mateusz Kwaśnicki}
\thanks{Maciej Kucharski was funded by FCT/Portugal and the Recovery and Resilience Plan (PRR) through projects UID/04459/2025 and UID/PRR/04459/2025, and by the project 2023.17881.ICDT (SHADE). Mateusz Kwaśnicki was supported by the National Science Centre, Poland, grant no.\@ 2023/49/B/ST1/04303}
\address{Maciej Kucharski \\ Mathematical Institute \\ University of Wrocław \\ Plac Grunwaldzki 2 \\ 50-384 Wrocław, Poland \& Centro de Análise Matemática, Geometria e Sistemas Dinâmicos, Instituto Superior Técnico, Av. Rovisco Pais, 1049-001 Lisboa, Portugal}
\email{\href{mailto:maciej.kucharski@math.uni.wroc.pl}{\textsf{maciej.kucharski@math.uni.wroc.pl}}}
\address{Mateusz Kwaśnicki \\ Department of Analysis and Stochastic Processes \\ Wrocław University of Science and Technology \\ Wybrzeże Wyspiańskiego 27 \\ 50-370 Wrocław, Poland}
\email{\href{mailto:mateusz.kwasnicki@pwr.edu.pl}{\textsf{mateusz.kwasnicki@pwr.edu.pl}}}
\keywords{Riesz transform, maximal inequality, dimension-free estimate}
\subjclass[2020]{%
 42B15, 
 42B20, 
 42B25
}

\begin{abstract}
We provide a shorter and more direct proof of $L^p$ estimates for maximal Riesz transforms (of an arbitrary order) in terms of the corresponding Riesz transforms, with a constant independent of the dimension of the Euclidean space $\R^d$. This result was originally proved by Mateu, Orobitg, Pérez and Verdera with a constant depending on the dimension, and improved to a dimension-free inequality by Kucharski, Wróbel and Zienkiewicz.
\end{abstract}

\maketitle

%
%

\section{Introduction}

The goal of this paper is to provide a shorter and more direct proof of the following maximal inequality for $n$th order Riesz transforms on Euclidean spaces of dimension $d \ge 1$: if $p \in (1, \infty)$ and $P$ is a solid (i.e.\@ homogeneous) harmonic polynomial on $\R^d$ of degree $n$, then for every Schwartz function $f$ we have
\formula[eq:maximal:riesz]{
 \lV R_*^{(P)} f \rV_p & \le C \lV R^{(P)} f \rV_p ,
}
with a constant $C$ which depends on $p$ and $n$, but not on the dimension $d$ nor on the particular choice of $P$. The $n$th order Riesz transform is a singular integral operator given by
\formula{
 R^{(P)} f(x) & = \frac{\Gamma(\tfrac{d + n}{2})}{\pi^{d / 2} \Gamma(\tfrac{n}{2})} \lim_{s \to 0^+} \int_{\R^d \setminus \ball(0, s)} \frac{P(y)}{\lv y \rv^{d + n}} \, f(x - y) dy ,
}
and the associated maximal function is given by
\formula{
 R_*^{(P)} f(x) & = \frac{\Gamma(\tfrac{d + n}{2})}{\pi^{d / 2} \Gamma(\tfrac{n}{2})} \sup_{s \in (0, \infty)} \biggl\lv \int_{\R^d \setminus \ball(0, s)} \frac{P(y)}{\lv y \rv^{d + n}} \, f(x - y) dy \biggr\rv .
}
The maximal inequality~\eqref{eq:maximal:riesz} with a constant $C$ which additionally depends on the dimension was proved by Mateu and Verdera in~\cite{mv} (when $n = 1$), Mateu, Orobitg and Verdera in~\cite{mov} (when $n$ is even), and Mateu, Orobitg, Pérez and Verdera in~\cite{mopv} (when $n$ is odd). The dimension-free bound stated above was found by Wróbel and the first named author in~\cite{kw} (when $n = 1$ and $p = 2$), Liu, Melentijević and Zhu in~\cite{lmz} (when $n = 1$), and Wróbel, Zienkiewicz and the first named author in~\cite{kwz}.

If $\fourier$ denotes the Fourier transform, then
\formula{
 \fourier\bigl[R^{(P)} f\bigr](\xi) & = \frac{P(-i \xi)}{\lv \xi \rv^n} \, \fourier f(\xi) = (-i)^n P\biggl(\frac{\xi}{\lv \xi \rv}\biggr) \fourier f(\xi) .
}
The truncated Riesz transforms
\formula{
 R_s^{(P)} f(x) & = \frac{\Gamma(\tfrac{d + n}{2})}{\pi^{d / 2} \Gamma(\tfrac{n}{2})} \int_{\R^d \setminus \ball(0, s)} \frac{P(y)}{\lv y \rv^{d + n}} \, f(x - y) dy
}
satisfy
\formula[eq:factorisation]{
 \fourier\bigl[R_s^{(P)} f\bigr](\xi) & = \fourier B_s^{(n)}(\xi) \, \fourier\bigl[R^{(P)} f\bigr](\xi) ,
}
where
\formula{
 B_s^{(n)}(x) & = s^{-d} B^{(d, n)}(s^{-1} \lv x \rv)
}
is an integrable radial function. Equivalently, the truncated Riesz transform is the composition of the Riesz transform and the convolution operator with integrable kernel $B_{\smash{s}}^{(n)}$. This appears implicitly in~\cite{mopv,mov,mv}, and it was explicitly used in~\cite{kw,lmz} (when $n = 1$) and~\cite{kwz} (for all $n$); see Proposition~2.1 in~\cite{kwz}. It follows that if $B_{\smash{*}}^{(n)} f$ is the maximal function defined by
\formula{
 B_*^{(n)} f(x) & = \sup_{s \in (0, \infty)} \lv f * B_s^{(n)}(x) \rv ,
}
then
\formula{
 R_*^{(P)} f(x) & \le B_*^{(n)}\bigl[R^{(P)} f\bigr](x) .
}
In particular, an $L^p(\R^d)$ estimate for $B_{\smash{*}}^{(n)}$ of the form $\lV B_{\smash{*}}^{(n)} f \rV_p \le C \lV f \rV_p$ implies the maximal inequality~\eqref{eq:maximal:riesz} with the same constant $C$, uniformly with respect to the choice of the polynomial $P$ of degree $n$. The following dimension-free bound for $B_{\smash{*}}^{(n)}$ from~\cite{kwz} thus yields~\eqref{eq:maximal:riesz} with a constant independent of the dimension.

\begin{theorem}[Theorems~2.2 and~2.3 in~\cite{kwz}]
\label{thm:maximal}
For every $n \ge 1$ and $p \in (1, \infty)$, there is a constant $C_{n, p}$ such that for $d \ge 1$ and a Schwartz function $f$ on $\R^d$, we have
\formula[eq:maximal]{
 \lV B_*^{(n)} f \rV_p & \le C_{n, p} \lV f \rV_p .
}
More generally, for every $n \ge 1$ and $p, q \in (1, \infty)$, there is a constant $C_{n, p, q}$ such that for $d \ge 1$ and Schwartz functions $f_1, f_2, \ldots, f_L$ on $\R^d$, we have the Fefferman--Stein bound
\formula[eq:vector]{
 \biggl\lV \biggl( \sum_{l = 1}^L (B_*^{(n)} f_l)^q \biggr)^{1/q} \biggr\rV_p & \le C_{n, p, q} \biggl\lV \biggl( \sum_{l = 1}^L \lv f_l \rv^q \biggr)^{1 / q} \biggr\rV_p .
}
\end{theorem}

We remark that the result of~\cite{kwz} only covers $q = 2$ in~\eqref{eq:vector}.

The papers cited above use a variety of advanced methods of harmonic analysis to prove~\eqref{eq:maximal:riesz} and~\eqref{eq:maximal}. Below we give a far more elementary and direct proof of~\eqref{eq:maximal}, which relies on a maximal inequality for Stein's generalised spherical means from~\cite{stein}, well-known properties of Gauss's hypergeometric function $_2F_1$, and the following explicit expression for $B^{(d, n)}$, found by Wróbel and the authors in~\cite{kkw}. Our proof of~\eqref{eq:vector} is very similar, except that Stein's maximal inequality is replaced by the corresponding Fefferman--Stein bound. To our knowledge, this vector-valued estimate does not appear in the literature, and so in Section~\ref{sec:vector} we provide a short argument, which follows closely the work of Deleaval and Kriegler~\cite{dk} on the Hardy--Littlewood maximal operator.

\begin{figure}[t]	
\begin{tabular}{cc}
\includegraphics[width=0.6\textwidth]{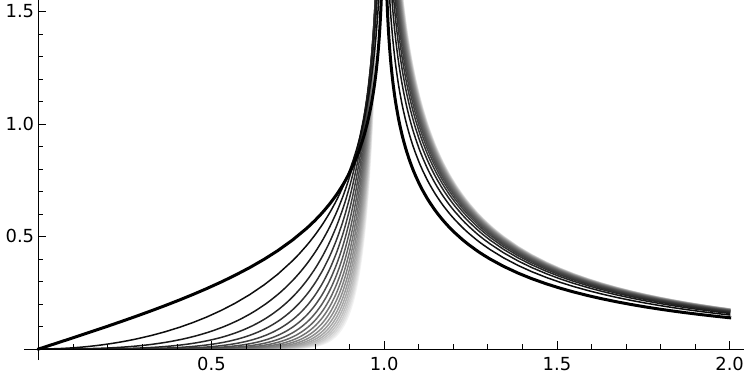} &
\includegraphics[width=0.3\textwidth]{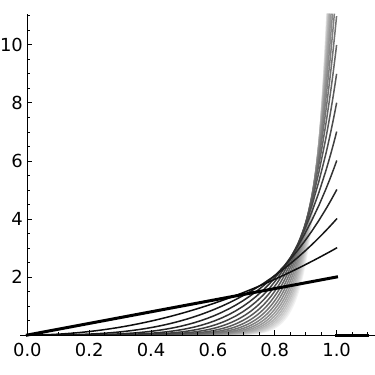} \\
\footnotesize $n = 1$ & \footnotesize $n = 2$ \\[4pt]
\includegraphics[width=0.6\textwidth]{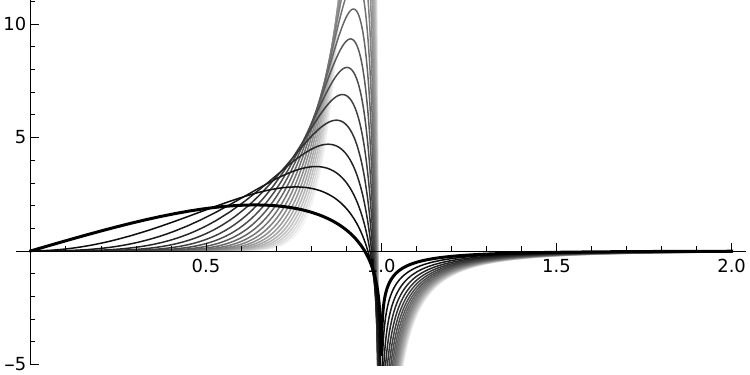} &
\includegraphics[width=0.3\textwidth]{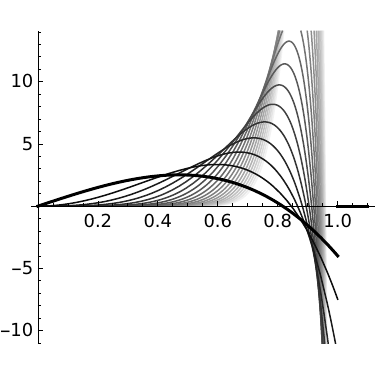} \\
\footnotesize $n = 3$ & \footnotesize $n = 4$ \\[4pt]
\includegraphics[width=0.6\textwidth]{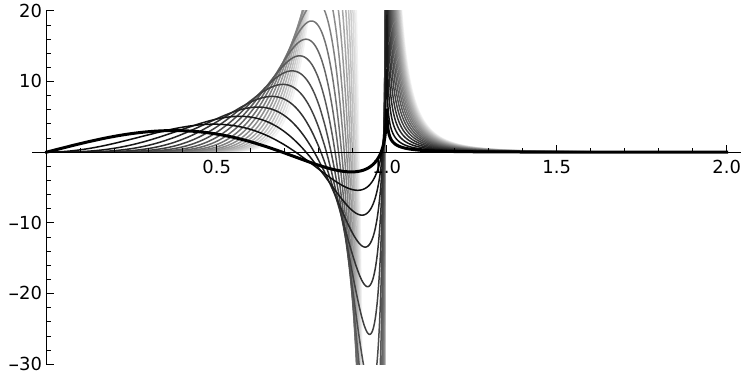} &
\includegraphics[width=0.3\textwidth]{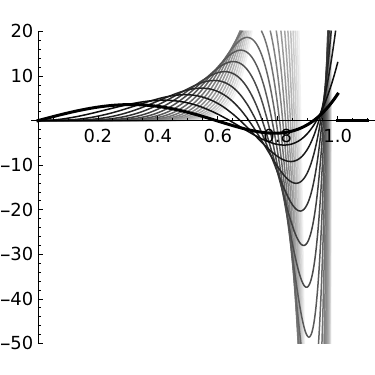} \\
\footnotesize $n = 5$ & \footnotesize $n = 6$ \\[4pt]
\includegraphics[width=0.6\textwidth]{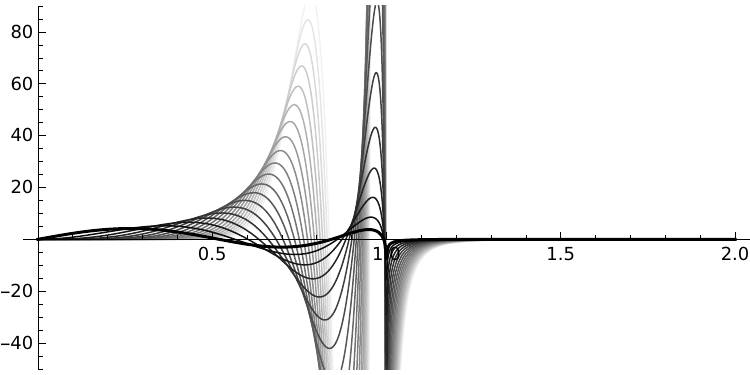} &
\includegraphics[width=0.3\textwidth]{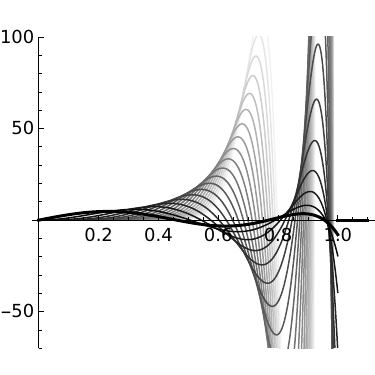} \\
\footnotesize $n = 7$ & \footnotesize $n = 8$
\end{tabular}
\caption{Volume-adjusted plots of $B^{(d, n)}(r)$ for a given $n$ and $d = 2, 3, \ldots, 21$. Black line corresponds to $d = 2$, and the brightness increases with $d$. The area under a curve corresponds to the integral of the corresponding $d$-dimensional kernel; that is, the plotted lines are the graphs of $2 \pi^{d / 2} (\Gamma(\tfrac{d}{2}))^{-1} r^{d - 1} B^{(d, n)}(r)$.}
\label{fig}
\end{figure}

\begin{proposition}[Proposition~2.2 in~\cite{kkw}]
\label{prop:kernel}
Let $d \ge 1$ and $n \ge 1$. If\/ $r \in (0, 1)$, then
\formula[eq:kernel:1]{
 B^{(d, n)}(r) & = \frac{(\Gamma(\tfrac{d + n}{2}))^2}{\pi^{d / 2} \Gamma(\tfrac{d}{2} + 1) (\Gamma(\tfrac{n}{2}))^2} \, \F{\tfrac{d + n}{2}}{1 - \tfrac{n}{2}}{\tfrac{d}{2} + 1}{r^2} ,
}
and in particular $B^{(d, n)}$ is equal to a polynomial on $(0, 1)$ for even $n$. If\/ $r \in (1, \infty)$, then we have $B^{(d, n)}(r) = 0$ for even $n$, and
\formula[eq:kernel:2]{
 B^{(d, n)}(r) & = \frac{(-1)^{(n - 1) / 2} (\Gamma(\tfrac{d + n}{2}))^2}{\pi^{d / 2 + 1} \Gamma(\tfrac{d}{2} + n)} \, \frac{1}{r^{d + n}} \, \F{\tfrac{d + n}{2}}{\tfrac{n}{2}}{\tfrac{d}{2} + n}{\frac{1}{r^2}}
}
when $n$ is odd (see Figure~\ref{fig}).
\end{proposition}

The proof of Proposition~\ref{prop:kernel} given in~\cite{kkw} is fairly elementary and direct. In Proposition~2.1 therein, Bochner's relation is used to determine the Fourier symbol of the truncated Riesz transform $R_{\smash{s}}^{(P)}$, and so, via~\eqref{eq:factorisation}, the Fourier transform of the kernel $B_{\smash{s}}^{(n)}$. Then, in Proposition~2.2, well-known properties of Bessel and hypergeometric functions are applied to invert the Fourier transform.

We refer to~\cite{kkw} for a more detailed discussion, including the impossibility of proving Theorem~\ref{thm:maximal} by estimating $B_{\smash{*}}^{(n)}$ using maximal operators which are bounded on $L^\infty(\R^d)$, fine properties of the kernel $B_{\smash{s}}^{(n)}$, and contractivity of $(R^{(P)})^{-1} R_{\smash{s}}^{(P)}$ (that is, the convolution operator with kernel $B_{\smash{s}}^{(n)}$) on $L^2(\R^d)$.

\begin{notation}
Throughout the paper we assume that $d \ge 1$ is the dimension, $n \ge 1$ is the order of the Riesz transform, and $p \in (1, \infty)$. Note that if $d = 1$, then there is only one first-order Riesz transform, the Hilbert transform, and there are no higher-order Riesz transforms. Thus, for $n \ge 2$ we only consider $d \ge 2$.

By $C$ we denote a generic positive constant. If $C$ depends on parameters, such as $d$, $n$ or $p$, we always list them in the subscript. The value of $C$ may change even within a single equation or inequality, so, for example, $C + C = C$.

By $\ball(x, r)$ we denote the Euclidean ball with radius $r$, centred at $x$, and $\sph$ denotes the unit sphere in $\R^d$. In inline equations, we write ${_2F_1}(a, b; c; t)$ for the hypergeometric function. We denote by $B_{\smash{s}}^{(n)}$ the kernel on $\R^d$ introduced in~\eqref{eq:factorisation}, and by $B^{(d, n)}$ the radial profile of $B_{\smash{1}}^{(n)}$ (denoted by $b_n$ in~\cite{kwz}). An explicit expression for $B^{(d, n)}$ is given in Proposition~\ref{prop:kernel}.

The Fourier transform of an integrable function $f$ on $\R^d$ is defined by
\formula{
 \fourier f(\xi) & = \int_{\R^d} e^{i \xi \cdot x} f(x) dx ,
}
and the operator $\fourier$ is extended continuously to $L^2(\R^d)$.
\end{notation}

\begin{structure}
In Section~\ref{sec:stein} we recall the definition of Stein's generalised spherical means $M_{\smash{r}}^{(\alpha)}$ (denoted $B_{\smash{r}}^\alpha$ in~\cite{stein}), and discuss the closely related concept of derivatives of spherical means $A_{\smash{r}}^{(k)}$. We also recall Stein's maximal inequality for $M_{\smash{r}}^{(\alpha)}$ and $A_{\smash{r}}^{(k)}$.

The proof of Theorem~\ref{thm:maximal} is given in Section~\ref{sec:proof}. More precisely, in Sections~\ref{sec:ball} and~\ref{sec:complement} we provide other expressions for the kernel $B^{(d, n)}$ in $(0, 1)$ and in $(1, \infty)$, respectively. Positivity of a certain hypergeometric function is studied in Section~\ref{sec:positivity}. Next, we prove Theorem~\ref{thm:maximal} for even $n$ and large $d$ in Section~\ref{sec:even:high}, and for even $n$ and small $d$ in Section~\ref{sec:even:low}. Similarly, Section~\ref{sec:odd:high} contains the proof of Theorem~\ref{thm:maximal} for odd $n$ and large $d$, and the case of odd $n$ and small $d$ is worked out in Section~\ref{sec:odd:low}.

In Section~\ref{sec:vector} we prove the Fefferman--Stein inequality for generalised spherical means (Theorem~\ref{thm:vector:generalised}) and then we use Theorem~\ref{thm:vector:generalised} and some results from \cite{k} to prove the Fefferman--Stein inequality for derivatives of spherical means (Theorem~\ref{thm:vector:stein}).
\end{structure}

%
%

\section{Stein's maximal inequality for derivatives of spherical averages}
\label{sec:stein}


For a Schwartz function $f$ on $\R^d$ and $r > 0$ we denote by $A_r f(x)$ the spherical average of~$f$:
\formula{
 A_r f(x) & = \int_{\sph} f(x + r y) \sigma(d y) ,
}
where $\sigma$ is the normalised surface measure. For $k = 0, 1, 2, \ldots$ let $A_r^{(k)} f$ denote the $k$th order derivative of the spherical mean with respect to the radius $r$, normalised by the factor $r^k$:
\formula[eq:derivatives]{
 A_r^{(k)} f(x) & = r^k \biggl(\frac{d}{dr}\biggr)^k \bigl[A_r f(x)\bigr] .
}
Then the corresponding maximal function
\formula{
 A_*^{(k)} f(x) & = \sup \bigl\{ \lv A_r^{(k)} f(x) \rv : r \in (0, \infty) \bigr\}
}
is bounded on $L^p(\R^d)$ with a constant independent of the dimension, that is,
\formula[eq:maximal:stein]{
 \lV A_*^{(k)} f \rV_p & \le C_{k, p} \lV f \rV_p ,
}
provided that
\formula[eq:range:stein]{
 k & = 1, 2, \ldots, \qquad d \ge 2 k + 3 , \qquad \frac{d}{d - k - 1} < p < \frac{d - 2}{k} .
}
The above estimate was essentially proved by Stein in~\cite{stein}. Later, Bourgain in~\cite{bourgain}, Bourgain and Demeter in~\cite{bd}, and Miao, Yang and Zheng in~\cite{myz} extended the range of parameters for which~\eqref{eq:maximal:stein} holds to
\formula[eq:range:bourgain:demeter]{
 k & = 1, 2, \ldots, \qquad d \ge 2 k + 3 , \qquad \frac{d}{d - k - 1} < p < \frac{d - 1}{k} ,
}
or $k = 0$, $d \ge 2$ and $p > \tfrac{d}{d - 1}$.

More precisely, the authors cited above studied Stein's generalised spherical means $M_{\smash{r}}^{(\alpha)}$, defined by
\formula[eq:generalised]{
 M_r^{(\alpha)} f(x) & = \frac{1}{\Gamma(\alpha)} \int_{\ball(0, 1)} f(x + r y) (1 - \lv y \rv^2)^{\alpha - 1} dy
}
when $\alpha > 0$, and extended to arbitrary complex $\alpha$ by holomorphic continuation. Stein's result (Theorem~2 in~\cite{stein}) in fact states that the associated maximal function
\formula{
 M_*^{(\alpha)} f(x) & = \sup_{r \in (0, \infty)} \lv M_r^{(\alpha)} f(x) \rv
}
satisfies the maximal inequality
\formula[eq:maximal:generalised]{
 \lV M_*^{(\alpha)} f \rV_p & \le C_{d, \alpha, p} \lV f \rV_p
}
for every Schwartz function $f$ on $\R^d$, provided that
\formula[eq:range:generalised]{
 d & \ge 1 , \qquad p \in (1, \infty] , \qquad \alpha > -\min\biggl\{d - 1 - \frac{d}{p} \, , \, \frac{d - 2}{p}\biggr\} ,
}
with the convention that $1 / \infty = 0$. The range of admissible $\alpha$ was appropriately extended by Bourgain, Demeter, Miao, Yang and Zheng.

Although not identical, the operators $M_{\smash{r}}^{(-k)}$ are closely related to the derivatives $A_{\smash{r}}^{(k)}$ of spherical means defined in~\eqref{eq:derivatives}. This connection is discussed in detail in the companion paper~\cite{k}, where~\eqref{eq:maximal:stein} with assumptions~\eqref{eq:range:stein} is given as Theorem~I, and the extension to~\eqref{eq:range:bourgain:demeter} is Theorem~II.

%
%

\section{Proof of Theorem~\texorpdfstring{\ref{thm:maximal}}{1}}
\label{sec:proof}

Integration in spherical coordinates allows us to express the convolution with the kernel $B_{\smash{s}}^{(n)}(x) = s^{-d} B^{(d, n)}(s^{-1} \lv x \rv)$ in terms of the spherical means:
\formula*[eq:averaging]{
 B_s^{(n)} * f(x) & = s^{-d} \int_{\R^d} f(x + y) B^{(d, n)}(s^{-1} \lv y \rv) dy \\
 & = \int_{\R^d} f(x + s z) B^{(d, n)}(\lv z \rv) dz \\
 & = \frac{2 \pi^{d / 2}}{\Gamma(\tfrac{d}{2})} \int_0^\infty r^{d - 1} B^{(d, n)}(r) A_{s r} f(x) dr \\
 & = \frac{\pi^{d / 2}}{\Gamma(\tfrac{d}{2})} \int_0^\infty t^{d / 2 - 1} B^{(d, n)}(\sqrt{t}) A_{s \sqrt{t}} f(x) dt
}
for every bounded Borel function $f$. In our proof of Theorem~\ref{thm:maximal}, we express the right-hand side in terms of the derivatives of spherical means of $f$.


\subsection{The kernel in the ball}
\label{sec:ball}

By \nist{15.5}{4} in~\cite{nist} we have
\formula{
 \frac{1}{\Gamma(c + k)} \, \biggl(\frac{d}{d t}\biggr)^k \biggl[ t^{c + k - 1} \F{a}{b}{c + k}{t} \biggr] & = \frac{1}{\Gamma(c)} \, t^{c - 1} \F{a}{b}{c}{t} .
}
Using this identity with $a = \tfrac{d + n}{2}$, $b = 1 - \tfrac{n}{2}$ and $c = \tfrac{d}{2} + 1$, we transform the expression~\eqref{eq:kernel:1} for the kernel $B^{(d, n)}$ to
\formula[eq:kernel:3]{
 B^{(d, n)}(\sqrt{t}) & = \frac{(\Gamma(\tfrac{d + n}{2}))^2}{\pi^{d / 2} \Gamma(\tfrac{d}{2} + k + 1) (\Gamma(\tfrac{n}{2}))^2} \, \frac{1}{t^{d / 2}} \biggl(\frac{d}{d t}\biggr)^k \biggl[t^{d / 2 + k} \F{\tfrac{d + n}{2}}{1 - \tfrac{n}{2}}{\tfrac{d}{2} + k + 1}{t}\biggr]
}
for $t \in (0, 1)$.

When $n \ge 2$ is even, we have $B^{(d, n)}(r) = 0$ for $r \in (1, \infty)$, and for $r \in (0, 1)$, we use~\eqref{eq:kernel:3} with $k = \tfrac{n}{2} - 1 \ge 0$. Substituting $n = 2 k + 2$, we obtain
\formula{
 B^{(d, n)}(\sqrt{t}) & = \frac{\Gamma(\tfrac{d}{2} + k + 1)}{\pi^{d / 2} (\Gamma(k + 1))^2} \, \frac{1}{t^{d / 2}} \biggl(\frac{d}{d t}\biggr)^k \biggl[t^{d / 2 + k} \F{\tfrac{d}{2} + k + 1}{-k}{\tfrac{d}{2} + k + 1}{t}\biggr]
}
for $t \in (0, 1)$. Since ${_2F_1}(a, b; a; t) = (1 - t)^{-b}$ (\nist{15.4}{6} in~\cite{nist}), we have
\formula[eq:kernel:4]{
 B^{(d, n)}(\sqrt{t}) & = \frac{\Gamma(\tfrac{d}{2} + k + 1)}{\pi^{d / 2} (\Gamma(k + 1))^2} \, \frac{1}{t^{d / 2}} \biggl(\frac{d}{d t}\biggr)^k \bigl[t^{d / 2 + k} (1 - t)^k \bigr] \ind_{(0, 1)}(t).
}

For odd $n \ge 1$, we substitute $k = \tfrac{n - 1}{2} \ge 0$, or $n = 2 k + 1$, into~\eqref{eq:kernel:3}: for $t \in (0, 1)$,
\formula{
 B^{(d, n)}(\sqrt{t}) & = \frac{(\Gamma(\tfrac{d}{2} + k + \tfrac{1}{2}))^2}{\pi^{d / 2} \Gamma(\tfrac{d}{2} + k + 1) (\Gamma(k + \tfrac{1}{2}))^2} \times {} \\
 & \qquad \times \frac{1}{t^{d / 2}} \biggl(\frac{d}{d t}\biggr)^k \biggl[t^{d / 2 + k} \F{\tfrac{d}{2} + k + \tfrac{1}{2}}{\tfrac{1}{2} - k}{\tfrac{d}{2} + k + 1}{t}\biggr] .
}
By \nist{15.8}{1} in~\cite{nist}, we have
\formula[eq:hypergeometric:duality]{
 \F{a}{b}{c}{t} & = (1 - t)^{c - a - b} \F{c - a}{c - b}{c}{t} ,
}
and hence
\formula*[eq:kernel:5]{
 B^{(d, n)}(\sqrt{t}) & = \frac{(\Gamma(\tfrac{d}{2} + k + \tfrac{1}{2}))^2}{\pi^{d / 2} \Gamma(\tfrac{d}{2} + k + 1) (\Gamma(k + \tfrac{1}{2}))^2} \times {} \\
 & \qquad \times \frac{1}{t^{d / 2}} \biggl(\frac{d}{d t}\biggr)^k \biggl[t^{d / 2 + k} (1 - t)^k \F{\tfrac{1}{2}}{\tfrac{d}{2} + 2 k + \tfrac{1}{2}}{\tfrac{d}{2} + k + 1}{t}\biggr]
}
for $t \in (0, 1)$.


\subsection{The kernel in the complement of the ball}
\label{sec:complement}

Suppose that $n \ge 1$ is odd, $n = 2 k + 1$. The hypergeometric function ${_2F_1}(a, b; c; t)$ extends to a holomorphic function of $t \in \C \setminus [1, \infty)$, with a continuous boundary limit on $(1, \infty)$ when approached from the upper complex half-plane (see~\href{https://dlmf.nist.gov/15.2.ii}{\S 15.2(ii)} in~\cite{nist}). By \nist{15.10}{25} (or \nist{15.8}{2}) in~\cite{nist}, for $t \in \C$ with $\im t > 0$, we have
\formula*[eq:hypergeometric:reciprocity]{
 \F{a}{b}{c}{t} & = \frac{\Gamma(c) \Gamma(b - a)}{\Gamma(b) \Gamma(c - a)} \, \frac{e^{i a \pi}}{t^a} \, \F{a}{a - c + 1}{a - b + 1}{\frac{1}{t}} \\
 & \qquad + \frac{\Gamma(c) \Gamma(a - b)}{\Gamma(a) \Gamma(c - b)} \, \frac{e^{i b \pi}}{t^b} \, \F{b}{b - c + 1}{b - a + 1}{\frac{1}{t}} ,
}
where complex powers assume principal branches. The right-hand side extends continuously to $t \in (1, \infty)$, providing an expression for the boundary limit ${_2F_1}(a, b; c; t + 0 i)$ approached from the upper complex half-plane.

Note that~\eqref{eq:hypergeometric:reciprocity} is valid only when $a - b$ is not an integer, otherwise the right-hand side is not well-defined; we refer to \nist{15.8}{8} in~\cite{nist} for a detailed discussion.

If $b = 1 - \tfrac{n}{2}$ and $n$ is odd, $a, c \in \R$, and $a + \tfrac{n}{2}$ is not an integer, then the second summand on the right-hand side of~\eqref{eq:hypergeometric:reciprocity} is purely imaginary for $t \in (1, \infty)$. In this case,
\formula{
 \re \F{a}{1 - \tfrac{n}{2}}{c}{t + 0 i} & = \frac{\Gamma(c) \Gamma(1 - a - \tfrac{n}{2})}{\Gamma(1 - \tfrac{n}{2}) \Gamma(c - a)} \, \frac{\cos(a \pi)}{t^a} \, \F{a}{a - c + 1}{a + \tfrac{n}{2}}{\frac{1}{t}} .
}
By Euler's reflection formula $\Gamma(z) \Gamma(1 - z) = \pi / \sin(\pi z)$ (\nist{5.5}{3} in~\cite{nist}) and periodicity,
\formula{
 \Gamma(1 - \tfrac{n}{2} - a) \cos(a \pi) & = \frac{\pi \cos(a \pi)}{\Gamma(a + \tfrac{n}{2}) \sin((a + \tfrac{n}{2}) \pi)} = \frac{(-1)^{(n - 1) / 2} \pi}{\Gamma(a + \tfrac{n}{2})} .
}
It follows that for $t \in (1, \infty)$,
\formula[eq:hypergeometric:reciprocity:real]{
 \re \F{a}{1 - \tfrac{n}{2}}{c}{t + 0 i} & = \frac{(-1)^{(n - 1) / 2} \pi \, \Gamma(c)}{\Gamma(a + \tfrac{n}{2}) \Gamma(1 - \tfrac{n}{2}) \Gamma(c - a)} \, \frac{1}{t^a} \, \F{a}{a - c + 1}{a + \tfrac{n}{2}}{\frac{1}{t}} .
}
While we proved this identity only when $a + \tfrac{n}{2}$ is not an integer, both sides define continuous functions of $a \in \R$, and so in fact the equality holds for all $a \in \R$.

We use~\eqref{eq:hypergeometric:reciprocity:real} with $a = \tfrac{d + n}{2}$ and $c = \tfrac{d}{2} + 1$: for $t \in (1, \infty)$,
\formula{
 \re \F{\tfrac{d + n}{2}}{1 - \tfrac{n}{2}}{\tfrac{d}{2} + 1}{t + 0 i} & = \frac{(-1)^{(n - 1) / 2} \pi \, \Gamma(\tfrac{d}{2} + 1)}{(\Gamma(1 - \tfrac{n}{2}))^2 \Gamma(\tfrac{d}{2} + n)} \, \frac{1}{t^{(d + n) / 2}} \, \F{\tfrac{d + n}{2}}{\tfrac{n}{2}}{\tfrac{d}{2} + n}{\frac{1}{t}} .
}
Therefore, the expression~\eqref{eq:kernel:2} for the kernel $B^{(d, n)}(r)$ for $r \in (1, \infty)$ can be written as
\formula{
 B^{(d, n)}(\sqrt{t}) & = \frac{(-1)^{(n - 1) / 2} (\Gamma(\tfrac{d + n}{2}))^2}{\pi^{d / 2 + 1} \Gamma(\tfrac{d}{2} + n)} \, \frac{1}{t^{(d + n) / 2}} \, \F{\tfrac{d + n}{2}}{\tfrac{n}{2}}{\tfrac{d}{2} + n}{\frac{1}{t}} \\
 & = \frac{(\Gamma(1 - \tfrac{n}{2}))^2 (\Gamma(\tfrac{d + n}{2}))^2}{\pi^{d / 2 + 2} \Gamma(\tfrac{d}{2} + 1)} \, \re \F{\tfrac{d + n}{2}}{1 - \tfrac{n}{2}}{\tfrac{d}{2} + 1}{t + 0 i}
}
for $t \in (1, \infty)$. By Euler's reflection formula, $\Gamma(1 - \tfrac{n}{2}) \Gamma(\tfrac{n}{2}) = (-1)^{(n - 1) / 2} \pi$, and therefore
\formula[eq:kernel:6]{
 B^{(d, n)}(\sqrt{t}) & = \frac{(\Gamma(\tfrac{d + n}{2}))^2}{\pi^{d / 2} \Gamma(\tfrac{d}{2} + 1) (\Gamma(\tfrac{n}{2}))^2} \, \re \F{\tfrac{d}{2} + k + \tfrac{1}{2}}{\tfrac{1}{2} - k}{\tfrac{d}{2} + 1}{t + 0 i} .
}
The above argument proves this identity for $t \in (1, \infty)$, and by~\eqref{eq:kernel:1} it also holds for $t \in (0, 1)$. Hence, it is true in the full range $t > 0$, $t \ne 1$.

Equivalence of the two expressions~\eqref{eq:kernel:1} and~\eqref{eq:kernel:5} (with $r = \sqrt{t}$) for the kernel $B^{(d, n)}$ on the interval $(0, 1)$ implies that
\formula{
 & \frac{(\Gamma(\tfrac{d + n}{2}))^2}{\pi^{d / 2} \Gamma(\tfrac{d}{2} + 1) (\Gamma(\tfrac{n}{2}))^2} \, \F{\tfrac{d + n}{2}}{1 - \tfrac{n}{2}}{\tfrac{d}{2} + 1}{t} \\
 & \qquad = \frac{(\Gamma(\tfrac{d}{2} + k + \tfrac{1}{2}))^2}{\pi^{d / 2} \Gamma(\tfrac{d}{2} + k + 1) (\Gamma(k + \tfrac{1}{2}))^2} \, \frac{1}{t^{d / 2}} \biggl(\frac{d}{d t}\biggr)^k \biggl[t^{d / 2 + k} (1 - t)^k \F{\tfrac{1}{2}}{\tfrac{d}{2} + 2 k + \tfrac{1}{2}}{\tfrac{d}{2} + k + 1}{t}\biggr]
}
for $t \in (0, 1)$. By the uniqueness of the holomorphic extension, this identity holds for all $t \in \C \setminus [1, \infty)$. We take the real parts of both sides and pass to the boundary limit at $t \in (1, \infty)$. By~\eqref{eq:kernel:6}, the boundary limit of the left-hand side is equal to $B^{(d, n)}(\sqrt{t})$, and hence
\formula*[eq:kernel:7]{
 B^{(d, n)}(\sqrt{t}) & = \frac{(\Gamma(\tfrac{d}{2} + k + \tfrac{1}{2}))^2}{\pi^{d / 2} \Gamma(\tfrac{d}{2} + k + 1) (\Gamma(k + \tfrac{1}{2}))^2} \times {} \\
 & \qquad \times \frac{1}{t^{d / 2}} \biggl(\frac{d}{d t}\biggr)^k \biggl[t^{d / 2 + k} (1 - t)^k \re \F{\tfrac{1}{2}}{\tfrac{d}{2} + 2 k + \tfrac{1}{2}}{\tfrac{d}{2} + k + 1}{t + 0 i}\biggr]
}
for $t \in (1, \infty)$. However, by~\eqref{eq:kernel:5}, the above equality also holds for $t \in (0, 1)$, and so it is valid in the full range $t > 0$, $t \ne 1$.


\subsection{Positivity property of the kernel}
\label{sec:positivity}

We claim that for every $t > 0$, $t \ne 1$,
\formula[eq:kernel:positive]{
 (1 - t)^k \re \F{\tfrac{1}{2}}{\tfrac{d}{2} + 2 k + \tfrac{1}{2}}{\tfrac{d}{2} + k + 1}{t + 0 i} & \ge 0 .
}
When $t \in (0, 1)$, this inequality is an immediate consequence of the definition of the hypergeometric function as a power series (\nist{15.2}{1} in~\cite{nist}): all coefficients of the power series are nonnegative. If $t \in (1, \infty)$, we transform the left-hand side using~\eqref{eq:hypergeometric:reciprocity:real} with $n = 1$, $a = \tfrac{d}{2} + 2 k + \tfrac{1}{2}$ and $c = \tfrac{d}{2} + k + 1$:
\formula{
 & \re \F{\tfrac{1}{2}}{\tfrac{d}{2} + 2 k + \tfrac{1}{2}}{\tfrac{d}{2} + k + 1}{t + 0 i} \\
 & \qquad = \frac{\pi \, \Gamma(\tfrac{d}{2} + k + 1)}{\Gamma(\tfrac{d}{2} + 2 k + 1) \Gamma(\tfrac{1}{2}) \Gamma(\tfrac{1}{2} - k)} \, \frac{1}{t^{(d + 4 k + 1) / 2}} \, \F{\tfrac{d}{2} + 2 k + \tfrac{1}{2}}{k + \tfrac{1}{2}}{\tfrac{d}{2} + 2 k + 1}{\frac{1}{t}} .
}
Using Euler's reflection formula $\Gamma(\tfrac{1}{2} - k) \Gamma(k + \tfrac{1}{2}) = (-1)^k \pi$ and~\eqref{eq:hypergeometric:duality}, this simplifies to
\formula{
 & (1 - t)^k \re \F{\tfrac{1}{2}}{\tfrac{d}{2} + 2 k + \tfrac{1}{2}}{\tfrac{d}{2} + k + 1}{t + 0 i} \\
 & \qquad = (1 - t)^k \, \frac{(-1)^k \Gamma(k + \tfrac{1}{2}) \Gamma(\tfrac{d}{2} + k + 1)}{\sqrt{\pi} \, \Gamma(\tfrac{d}{2} + 2 k + 1)} \, \frac{1}{t^{(d + 4 k + 1) / 2}} \, \frac{1}{(1 - t^{-1})^k} \, \F{\tfrac{1}{2}}{\tfrac{d}{2} + k + \tfrac{1}{2}}{\tfrac{d}{2} + 2 k + 1}{\frac{1}{t}} \\
 & \qquad = \frac{\Gamma(k + \tfrac{1}{2}) \Gamma(\tfrac{d}{2} + k + 1)}{\sqrt{\pi} \, \Gamma(\tfrac{d}{2} + 2 k + 1)} \, \frac{1}{t^{(d + 2 k + 1) / 2}} \, \F{\tfrac{1}{2}}{\tfrac{d}{2} + k + \tfrac{1}{2}}{\tfrac{d}{2} + 2 k + 1}{\frac{1}{t}} .
}
The right-hand side is positive when $t \in (1, \infty)$ by the definition of the hypergeometric function, and our claim follows.


\subsection{Even order in high dimensions}
\label{sec:even:high}

Suppose that $n \ge 2$ is even, $n = 2 k + 2$, and let $f$ be a Schwartz function on $\R^d$. Combining formulae~\eqref{eq:averaging} and~\eqref{eq:kernel:4}, we find that
\formula{
 B_s^{(n)} * f(x) & = \frac{\Gamma(\tfrac{d}{2} + k + 1)}{\Gamma(\tfrac{d}{2}) (\Gamma(k + 1))^2} \int_0^1 \frac{1}{t} \biggl(\frac{d}{d t}\biggr)^k \bigl[t^{d / 2 + k} (1 - t)^k \bigr] A_{s \sqrt{t}} f(x) dt .
}
Integration by parts repeated $k$ times leads to
\formula{
 B_s^{(n)} * f(x) & = \frac{(-1)^k \Gamma(\tfrac{d}{2} + k + 1)}{\Gamma(\tfrac{d}{2}) (\Gamma(k + 1))^2} \int_0^1 t^{d / 2 + k} (1 - t)^k \biggl(\frac{d}{d t}\biggr)^k \biggl[\frac{1}{t} \, A_{s \sqrt{t}} f(x) \biggr] dt ;
}
since $A_{\smash{r}}^{(j)} f(x)$ has a finite limit as $r \to 0^+$ for every $j$, all boundary terms are easily proved to be zero. If we expand the derivatives, we get
\formula{
 B_s^{(n)} * f(x) & = \frac{\Gamma(\tfrac{d}{2} + k + 1)}{\Gamma(\tfrac{d}{2})} \int_0^1 t^{d / 2 - 1} (1 - t)^k \biggl(\sum_{j = 0}^k C_{k, j} A_{s \sqrt{t}}^{(j)} f(x)\biggr) dt
}
for appropriate constants $C_{k, j}$ (which also include the factor $(-1)^k (\Gamma(k + 1))^{-2}$). However, $A_{\smash{r}}^{(j)} f(x)$ is bounded by the corresponding maximal function $A_{\smash{*}}^{(j)} f(x)$. Thus,
\formula{
 \lv B_s^{(n)} * f(x) \rv & \le \frac{\Gamma(\tfrac{d}{2} + k + 1)}{\Gamma(\tfrac{d}{2})} \, \biggl(\int_0^1 t^{d / 2 - 1} (1 - t)^k dt\biggr) \biggl(\sum_{j = 0}^k C_{k, j} A_*^{(j)} f(x)\biggr) .
}
By the beta integral (\nist{5.12}{1} in~\cite{nist}),
\formula{
 \lv B_s^{(n)} * f(x) \rv & \le \Gamma(k + 1) \sum_{j = 0}^k C_{k, j} A_*^{(j)} f(x) .
}
The right-hand side no longer depends on $s > 0$, and so
\formula{
 B_*^{(n)} f(x) & \le \Gamma(k + 1) \sum_{j = 0}^k C_{k, j} A_*^{(j)} f(x) .
}
Using Stein's maximal inequality~\eqref{eq:maximal:stein}, we conclude that
\formula{
 \lV B_*^{(n)} f \rV_p & \le \Gamma(k + 1) \sum_{j = 0}^k C_{k, j} \lV A_*^{(j)} f \rV_p \le C_{k, p} \lV f \rV_p ,
}
as long as condition~\eqref{eq:range:bourgain:demeter} holds. This is true when $d$ is large enough, and so the desired maximal inequality~\eqref{eq:maximal} follows for $d \ge C_{k, p}$. (Note that $k = \tfrac{n}{2} - 1$, and so $C_{k, p} = C_{n, p}$.)

The same argument, with Stein's maximal inequality~\eqref{eq:maximal:stein} replaced by the Fefferman--Stein bound~\eqref{eq:vector:stein} from Theorem~\ref{thm:vector:stein}, proves~\eqref{eq:vector} for $d \ge C_{k, p, q}$: by the convexity of the norm and~\eqref{eq:vector:stein},
\formula{
 \biggl\lV \biggl( \sum_{l = 1}^L \lv B_*^{(n)} f_l \rv^q \biggr)^{1 / q} \biggr\rV_p & \le \Gamma(k + 1) \biggl\lV \biggl( \sum_{l = 1}^L \biggl( \sum_{j = 0}^k C_{k, j} A_*^{(j)} f_l \biggr)^q \biggr)^{1 / q} \biggr\rV_p \\
 & \le \Gamma(k + 1) \sum_{j = 0}^k C_{k, j} \biggl\lV \biggl( \sum_{l = 1}^L (A_*^{(j)} f_l)^q \biggr)^{1 / q} \biggr\rV_p \\
 & \le \Gamma(k + 1) C_{k, p, q} \biggl\lV \biggl( \sum_{l = 1}^L \lv f_l \rv^q \biggr)^{1 / q} \biggr\rV_p ,
}
as long as~\eqref{eq:range:vector} holds.


\subsection{Even order in low dimensions}
\label{sec:even:low}

In order to complete the proof of Theorem~\ref{thm:maximal} when $n$ is even, we need to prove the maximal inequalities~\eqref{eq:maximal} and~\eqref{eq:vector} in low dimensions, $1 \le d \le C_{n, p}$. For this purpose, it suffices to prove them with a constant $C_{d, n, p}$ that may also depend on the dimension.

This result is a consequence of a simple majorization of $B_{\smash{s}}^{(n)} * f(x)$ by the Hardy--Littlewood maximal function of $f$, which, up to multiplication by a constant, is equal to $M_{\smash{*}}^{(1)} f$. Indeed: the kernel $B^{(d, n)}$ is bounded on $(0, 1)$ and zero on $(1, \infty)$, and so
\formula{
 \lv B_s^{(n)} * f(x) \rv & = s^{-d} \int_{\R^d} f(x - y) B^{(d, n)}(s^{-1} \lv y \rv) dy \\
 & \le \biggl(\sup_{r \in (0, 1)} \lv B^{(d, n)}(r) \rv \biggr) \, \frac{1}{s^d} \int_{\ball(0, s)} \lv f(x - y) \rv dy \\
 & = \biggl(\sup_{r \in (0, 1)} \lv B^{(d, n)}(r) \rv \biggr) \, M_s^{(1)} f(x) \\
 & \le C_{d, n} M_*^{(1)} f(x)
}
for every Schwartz function $f$ on $\R^d$. Together with the Hardy--Littlewood maximal inequality (or~\eqref{eq:maximal:generalised}), this proves~\eqref{eq:maximal} with a dimension-dependent constant $C_{d, n, p}$. Similarly, the above estimate combined with the classical Fefferman--Stein bound~\eqref{eq:vector:hardy:littlewood} yields~\eqref{eq:vector} with a constant $C_{d, n, p}$ that depends on the dimension.

We remark that the above paragraph, together with the proof of Proposition~\ref{prop:kernel} given in~\cite{kkw}, provides a very short derivation of the maximal inequality~\eqref{eq:maximal:riesz} for even $n$ with a dimension-dependent constant $C$. This result was originally proved in~\cite{mov}.


\subsection{Odd order in high dimensions}
\label{sec:odd:high}

If $n \ge 1$ is odd, $n = 2 k + 1$, we follow the same approach, but we use~\eqref{eq:kernel:7} instead of~\eqref{eq:kernel:4}. Combining this formula and~\eqref{eq:averaging}, we find that for Schwartz functions $f$ on $\R^d$,
\formula{
 B_s^{(n)} * f(x) & = \frac{(\Gamma(\tfrac{d}{2} + k + \tfrac{1}{2}))^2}{\Gamma(\tfrac{d}{2}) \Gamma(\tfrac{d}{2} + k + 1) (\Gamma(k + \tfrac{1}{2}))^2} \times {} \\
 & \hspace*{-2em} \times \int_0^\infty \frac{1}{t} \biggl(\frac{d}{d t}\biggr)^k \biggl[t^{d / 2 + k} (1 - t)^k \re \F{\tfrac{1}{2}}{\tfrac{d}{2} + 2 k + \tfrac{1}{2}}{\tfrac{d}{2} + k + 1}{t + 0 i}\biggr] A_{s \sqrt{t}} f(x) dt .
}
As before, we integrate by parts $k$ times, and we get
\formula{
 B_s^{(n)} * f(x) & = \frac{(-1)^k (\Gamma(\tfrac{d}{2} + k + \tfrac{1}{2}))^2}{\Gamma(\tfrac{d}{2}) \Gamma(\tfrac{d}{2} + k + 1) (\Gamma(k + \tfrac{1}{2}))^2} \times {} \\
 & \hspace*{-2em} \times \int_0^\infty t^{d / 2 + k} (1 - t)^k \re \F{\tfrac{1}{2}}{\tfrac{d}{2} + 2 k + \tfrac{1}{2}}{\tfrac{d}{2} + k + 1}{t + 0 i} \biggl(\frac{d}{d t}\biggr)^k \biggl[\frac{1}{t} A_{s \sqrt{t}} f(x)\biggr] dt ,
}
but this time it is more difficult to control the boundary terms. In the $j$th integration by parts:
\begin{itemize}
\item the boundary term at $0$ is of order at most $t^{d/2 + j} t^{-j} = t^{d / 2}$, because $A_{\smash{r}}^{(j)} f(x)$ has a finite limit as $r \to 0^+$ for every $j$;
\item the boundary term at $\infty$ also vanishes, because $A_{\smash{r}}^{(j)} f$ is rapidly decaying as $r \to \infty$ for every $j$, while by~\eqref{eq:hypergeometric:reciprocity:real} the hypergeometric function
\formula{
 \Phi(t) & = \re \F{\tfrac{1}{2}}{\tfrac{d}{2} + 2 k + \tfrac{1}{2}}{\tfrac{d}{2} + k + 1}{t + 0 i} ,
}
as well as all its derivatives, has power-type growth as $t \to \infty$ (see~\href{https://dlmf.nist.gov/15.12.i}{\S 15.12(i)} in~\cite{nist});
\item the function $t^{d / 2 + k} (1 - t)^k \Phi(t)$ has $k - 1$ continuous derivatives at $t = 1$ (see \nist{15.4}{23} and \nist{15.5}{1} in~\cite{nist});
\item the $k$th derivative of this function involves the term $\Phi(t)$, which has a logarithmic singularity near $t = 1$ (see \nist{15.4}{21} in~\cite{nist}).
\end{itemize}
As in the case of even order $n$, we expand the derivatives:
\formula{
 B_s^{(n)} * f(x) & = \frac{(\Gamma(\tfrac{d}{2} + k + \tfrac{1}{2}))^2}{\Gamma(\tfrac{d}{2}) \Gamma(\tfrac{d}{2} + k + 1)} \times {} \\
 & \hspace*{-2em} \times \int_0^\infty t^{d / 2 - 1} (1 - t)^k \re \F{\tfrac{1}{2}}{\tfrac{d}{2} + 2 k + \tfrac{1}{2}}{\tfrac{d}{2} + k + 1}{t + 0 i} \biggl(\sum_{j = 0}^k C_{k, j} A_{s \sqrt{t}}^{(j)} f(x)\biggr) dt ,
}
and estimate $A_{\smash{r}}^{(j)} f(x)$ by the corresponding maximal function $A_{\smash{*}}^{(j)} f(x)$:
\formula*[eq:averaging:derivatives]{
 \lv B_s^{(n)} * f(x) \rv & \le \frac{(\Gamma(\tfrac{d}{2} + k + \tfrac{1}{2}))^2}{\Gamma(\tfrac{d}{2}) \Gamma(\tfrac{d}{2} + k + 1)} \times {} \\
 & \hspace*{-2em} \times \biggl(\int_0^\infty t^{d / 2 - 1} (1 - t)^k \re \F{\tfrac{1}{2}}{\tfrac{d}{2} + 2 k + \tfrac{1}{2}}{\tfrac{d}{2} + k + 1}{t + 0 i} dt\biggr) \biggl(\sum_{j = 0}^k C_{k, j} A_*^{(j)} f(x)\biggr) ;
}
this is where we use the positivity property~\eqref{eq:kernel:positive}.

It remains to evaluate the integral on the right-hand side of the above estimate. For this purpose, we repeat the above calculation for a constant function $f(x) = 1$. Since $B_{\smash{s}}^{(n)} * f(x) = 1$ and $A_{\smash{r}} f(x) = 1$, by~\eqref{eq:averaging} we obtain
\formula{
 1 & = \frac{(\Gamma(\tfrac{d}{2} + k + \tfrac{1}{2}))^2}{\Gamma(\tfrac{d}{2}) \Gamma(\tfrac{d}{2} + k + 1) (\Gamma(k + \tfrac{1}{2}))^2} \times {} \\
 & \qquad \times \int_0^\infty \frac{1}{t} \biggl(\frac{d}{d t}\biggr)^k \biggl[t^{d / 2 + k} (1 - t)^k \re \F{\tfrac{1}{2}}{\tfrac{d}{2} + 2 k + \tfrac{1}{2}}{\tfrac{d}{2} + k + 1}{t + 0 i}\biggr]  dt .
}
We integrate by parts $k$ times:
\formula{
 1 & = \frac{k! (\Gamma(\tfrac{d}{2} + k + \tfrac{1}{2}))^2}{\Gamma(\tfrac{d}{2}) \Gamma(\tfrac{d}{2} + k + 1) (\Gamma(k + \tfrac{1}{2}))^2} \int_0^\infty t^{d / 2 - 1} (1 - t)^k \re \F{\tfrac{1}{2}}{\tfrac{d}{2} + 2 k + \tfrac{1}{2}}{\tfrac{d}{2} + k + 1}{t + 0 i} dt .
}
Again all boundary terms vanish, but the argument for $t \to \infty$ is different: we no longer have a rapidly decaying term. However, by~\eqref{eq:hypergeometric:reciprocity:real}, in the $j$th integration by parts the boundary term at $\infty$ is of order at most $t^{-j} t^{-k + j - 1/2} = t^{-k - 1/2}$.

The above equality provides an expression for the integral that appears in~\eqref{eq:averaging:derivatives}. Combining the two formulae, we get the desired estimate
\formula{
 \lv B_s^{(n)} * f(x) \rv & \le \sum_{j = 0}^k C_{k, j} A_*^{(j)} f(x) ,
}
which, just as in the case of even order $n$, implies~\eqref{eq:maximal} for $d \ge C_{k, p}$. Once again, the Fefferman--Stein bound~\eqref{eq:vector} is obtained by replacing~\eqref{eq:maximal:stein} by~\eqref{eq:vector:stein} from Theorem~\ref{thm:vector:stein}.


\subsection{Odd order in low dimensions}
\label{sec:odd:low}

As in the case of the even order $n$, we complete the proof of Theorem~\ref{thm:maximal} for odd $n$ by proving the maximal inequalities~\eqref{eq:maximal} and~\eqref{eq:vector} with constants $C_{d, n, p}$ that may depend on the dimension. If $n$ is odd, however, the function $B^{(d, n)}$ is no longer bounded near $r = 1$, and the proof is necessarily more involved.  We use the fact that the singularity of $B^{(d, n)}$ is logarithmic, and so the convolutions $B_{\smash{s}}^{(n)} * f$ are dominated by appropriate mixtures of Stein's generalised fractional means $M_{\smash{r}}^{(\alpha)} \lv f \rv$ (as defined in~\eqref{eq:generalised}), where $\alpha \in (0, 1)$ can be arbitrarily close to $1$. More precisely, we claim that if $\delta \in (0, \tfrac{1}{2})$ and $f$ is a Schwartz function on $\R^d$, then
\formula[eq:odd:claim]{
 \lv B_s^{(n)} * f(x) \rv & \le C_{d, n, \delta} \int_1^\infty M_{s \sqrt{u}}^{(1 - 2 \delta)} \lv f \rv (x) (u - 1)^{\delta - 1} u^{-n/2 - \delta} du .
}
Once this claim is proved, we have
\formula{
 \lv B_s^{(n)} * f(x) \rv & \le C_{d, n, \delta} \biggl(\int_1^\infty (u - 1)^{\delta - 1} u^{-n/2 - \delta} du\biggr) M_*^{(1 - 2 \delta)} \lv f \rv (x) = C_{d, n, \delta} M_*^{(1 - 2 \delta)} \lv f \rv (x) .
}
The maximal inequality~\eqref{eq:maximal}, with a constant $C_{d, n, p}$ depending on the dimension, follows then from Stein's maximal inequality~\eqref{eq:maximal:generalised}, provided that we choose $\delta \in (0, \tfrac{1}{2})$ small enough (depending on $d$, $n$ and $p$), so condition~\eqref{eq:range:generalised} is satisfied with $\alpha = 1 - 2 \delta$. The same argument, but with Stein's maximal inequality~\eqref{eq:maximal:generalised} replaced by the corresponding Fefferman--Stein estimate~\eqref{eq:vector:generalised} from Theorem~\ref{thm:vector:generalised}, leads to~\eqref{eq:vector}.

It remains to prove our claim. This requires a simple, but somewhat lengthy estimate of the kernel. The behaviour of $B^{(d, n)}(r)$ near $r = 1$ is described by \nist{15.4}{21} in~\cite{nist}:
\formula[eq:hypergeometric:one:log]{
 \lim_{t \to 1^-} \frac{1}{-\log(1 - t)} \F{a}{b}{a + b}{t} & = \frac{\Gamma(a + b)}{\Gamma(a) \Gamma(b)} \, .
}
By this property and~\eqref{eq:kernel:1}, after simplification,
\formula{
 \lim_{t \to 1^-} \frac{B^{(d, n)}(\sqrt{t})}{-\log (1 - t)} & = \frac{(-1)^{(n - 1) / 2} \Gamma(\tfrac{d + n}{2})}{\pi^{d / 2 + 1} \Gamma(\tfrac{n}{2})} ,
}
and since the hypergeometric function is continuous on $[0, 1)$, it follows that for $t \in (0, 1)$,
\formula[eq:odd:1]{
 \lv B^{(d, n)}(\sqrt{t}) \rv & \le C_{d, n} (1 - \log (1 - t)) .
}
Similarly, by~\eqref{eq:kernel:2} and~\eqref{eq:hypergeometric:one:log},
\formula{
 \lim_{t \to 1^+} \frac{B^{(d, n)}(\sqrt{t})}{-\log (1 - t^{-1})} & = \frac{(-1)^{(n - 1) / 2} \Gamma(\tfrac{d + n}{2})}{\pi^{d / 2 + 1} \Gamma(\tfrac{n}{2})} ,
}
and therefore, for $t \in (1, \infty)$,
\formula[eq:odd:2]{
 \lv B^{(d, n)}(\sqrt{t}) \rv & \le C_{d, n} \, \frac{1}{t^{(d + n) / 2}} \, (1 - \log (1 - t^{-1})) .
}
We turn to the estimates of the kernel on the right-hand side of~\eqref{eq:odd:claim}. Let $\delta \in (0, \tfrac{1}{2})$. We need two more properties of the hypergeometric function. By \nist{15.6}{1} in~\cite{nist}, if $0 < b < c$, then, for $t \in (0, 1)$,
\formula[eq:hypergeometric:integral]{
 \F{a}{b}{c}{t} & = \frac{\Gamma(c)}{\Gamma(b) \Gamma(c - b)} \int_0^1 \frac{v^{b - 1} (1 - v)^{c - b - 1}}{(1 - t v)^a} \, dv .
}
By \nist{15.4}{23} in~\cite{nist},
\formula[eq:hypergeometric:one:power]{
 \lim_{t \to 1^-} \frac{1}{(1 - t)^{-\delta}} \F{a}{b}{a + b - \delta}{t} & = \frac{\Gamma(a + b - \delta) \Gamma(\delta)}{\Gamma(a) \Gamma(b)} .
}
Using the substitution $u = 1 / v$ and~\eqref{eq:hypergeometric:integral}, we have
\formula{
 \int_1^\infty (u - t)^{-2 \delta} (u - 1)^{\delta - 1} u^{-(d + n) / 2 + \delta} du & = \int_0^1 (1 - t v)^{-2 \delta} (1 - v)^{\delta - 1} v^{(d + n) / 2 - 1} dv \\
 & = \frac{\Gamma(\tfrac{d + n}{2}) \Gamma(\delta)}{\Gamma(\tfrac{d + n}{2} + \delta)} \, \F{2 \delta}{\tfrac{d + n}{2}}{\tfrac{d + n}{2} + \delta}{t}
}
for $t \in (0, 1)$. By~\eqref{eq:hypergeometric:one:power}, the right-hand side is of order $(1 - t)^{-\delta}$ as $t \to 1^-$, and so
\formula[eq:odd:3]{
 \int_1^\infty (u - t)^{-2 \delta} (u - 1)^{\delta - 1} u^{-(d + n) / 2 + \delta} du & \ge C_{d, n, \delta} \, \frac{1}{(1 - t)^\delta} .
}
Similarly, if $t \in (1, \infty)$, then, by the substitution $u = t / v$ and~\eqref{eq:hypergeometric:integral},
\formula{
 \int_t^\infty (u - t)^{-2 \delta} (u - 1)^{\delta - 1} u^{-(d + n) / 2 + \delta} du & = \frac{1}{t^{(d + n) / 2}} \int_0^1 (1 - v t^{-1})^{\delta - 1} (1 - v)^{-2 \delta} v^{(d + n) / 2 - 1} dv \\
 & = \frac{\Gamma(\tfrac{d + n}{2}) \Gamma(1 - 2 \delta)}{\Gamma(\tfrac{d + n}{2} + 1 - 2 \delta)} \, \frac{1}{t^{(d + n) / 2}} \, \F{1 - \delta}{\tfrac{d + n}{2}}{\tfrac{d + n}{2} + 1 - 2 \delta}{\frac{1}{t}} .
}
By~\eqref{eq:hypergeometric:one:power}, the right-hand side is of order $(1 - t^{-1})^{-\delta}$ as $t \to 1^+$, and so
\formula[eq:odd:4]{
 \int_t^\infty (u - t)^{-2 \delta} (u - 1)^{\delta - 1} u^{-(d + n) / 2 + \delta} du & \ge C_{d, n, \delta} \, \frac{1}{t^{(d + n) / 2}} \, \frac{1}{(1 - t^{-1})^\delta} .
}
Observe that $0 < 1 - \log(1 - t) \le C_\delta (1 - t)^{-\delta}$ for $t \in (0, 1)$. Hence, the above four estimates~\eqref{eq:odd:1}, \eqref{eq:odd:2}, \eqref{eq:odd:3} and \eqref{eq:odd:4} imply that for all $t > 0$, $t \ne 1$, we have
\formula{
 \lv B^{(d, n)}(\sqrt{t}) \rv & \le C_{d, n, \delta} \, \int_{\max\{1, t\}}^\infty (u - t)^{-2 \delta} (u - 1)^{\delta - 1} u^{-(d + n) / 2 + \delta} du .
}
Now~\eqref{eq:averaging}, the above inequality, Fubini's theorem, and the substitution $t = u w^2$ lead to
\formula{
 \lv B_s^{(n)} * f(x) \rv & = \frac{\pi^{d / 2}}{\Gamma(\tfrac{d}{2})} \biggl| \int_0^\infty t^{d / 2 - 1} B^{(d, n)}(\sqrt{t}) A_{s \sqrt{t}} f(x) dt \biggr| \\
 & \le \frac{\pi^{d / 2}}{\Gamma(\tfrac{d}{2})} \int_0^\infty t^{d / 2 - 1} \lv B^{(d, n)}(\sqrt{t}) \rv A_{s \sqrt{t}} \lv f \rv (x) dt \displaybreak[0] \\
 & \le C_{d, n, \delta} \int_0^\infty t^{d / 2 - 1} \biggl( \int_{\max\{1, t\}}^\infty (u - t)^{-2 \delta} (u - 1)^{\delta - 1} u^{-(d + n) / 2 + \delta} du \biggr) A_{s \sqrt{t}} \lv f \rv (x) dt \displaybreak[0] \\
 & \le C_{d, n, \delta} \int_1^\infty \biggl( \int_0^u t^{d / 2 - 1} (u - t)^{-2 \delta} A_{s \sqrt{t}} \lv f \rv (x) dt \biggr) (u - 1)^{\delta - 1} u^{-(d + n) / 2 + \delta} du \\
 & = C_{d, n, \delta} \int_1^\infty \biggl( \int_0^1 w^{d - 1} (1 - w^2)^{-2 \delta} A_{s w \sqrt{u}} \lv f \rv (x) dw \biggr) (u - 1)^{\delta - 1} u^{-n / 2 - \delta} du .
}
However, using the definition~\eqref{eq:generalised} of Stein's generalised spherical means and integration in spherical coordinates, we find that, just as in~\eqref{eq:averaging},
\formula{
 M_r^{(1 - 2 \delta)} \lv f \rv (x) & = \frac{1}{\Gamma(1 - 2 \delta)} \int_{\ball(0, 1)} \lv f(x + r y) \rv (1 - \lv y \rv^2)^{-2 \delta} dy \\
 & = \frac{2 \pi^{d / 2}}{\Gamma(1 - 2 \delta) \Gamma(\tfrac{d}{2})} \int_0^1 w^{d - 1} (1 - w^2)^{-2 \delta} A_{r w} \lv f \rv (x) dw ,
}
and our claim~\eqref{eq:odd:claim} follows.

We note that this section, together with the proof of Proposition~\ref{prop:kernel} given in~\cite{kkw}, provides a direct and relatively short proof of the maximal inequality~\eqref{eq:maximal:riesz} for odd $n$ with a dimension-dependent constant $C$, a result originally proved in~\cite{mopv}.

%
%

\section{Fefferman--Stein inequality for Stein's generalised spherical means}
\label{sec:vector}

In their seminal work~\cite{fs}, Fefferman and Stein proved the following vector-valued inequality for the Hardy--Littlewood maximal operator $M_{\smash{*}}^{(1)}$: if $p, q \in (1, \infty)$, then for all Schwartz functions $f_1, f_2, \ldots, f_L$,
\formula[eq:vector:hardy:littlewood]{
 \biggl\lV \biggl( \sum_{l = 1}^L (M_*^{(1)} f_l)^q \biggr)^{1 / q} \biggr\rV_p & \le C_{d, p, q} \biggl\lV \biggl( \sum_{l = 1}^L \lv f_l \rv^q \biggr)^{1 / q} \biggr\rV_p .
}
By continuity, the above estimate holds for infinite sequences $f_1, f_2, \ldots$ of functions in $L^p(\R^d)$.

A fully analogous inequality for Stein's spherical maximal operator $A_{\smash{*}}^{(0)}$, with a constant independent of the dimension, was given by Deleaval and Kriegler in~\cite{dk}, under the assumptions $d \ge 3$ and $p, q \in (\tfrac{d}{d - 1}, d)$. Their proof combines the Fefferman--Stein inequality~\eqref{eq:vector:hardy:littlewood} with the ideas of Bourgain's work~\cite{bourgain} on Stein's spherical maximal function, and then applies the method of rotations from Stein's work~\cite{stein:essay} (see also~\cite{ss}) to show that the constant improves with the dimension. After minor modifications, the method of~\cite{dk} applies to Stein's generalised spherical means $M_{\smash{r}}^{(\alpha)}$ and the derivatives of spherical means $A_{\smash{r}}^{(k)}$, leading to the following two results. For the reader's convenience, we provide complete proofs.

\begin{theorem}[Fefferman--Stein inequality for generalised spherical means]
\label{thm:vector:generalised}
Suppose that
\formula{
 & \alpha \in \C , \qquad d \ge 1 , \qquad \re \alpha > 1 - \tfrac{d}{2} , \qquad p, q \in (1, \infty) ,
}
and, if\/ $\re \alpha < 1$, additionally
\formula{
 p, q \in \biggl(\frac{d}{d - 1 + \re \alpha} , \frac{d}{1 - \re \alpha} \biggr) .
}
If\/ $f_1, f_2, \ldots, f_L$ are Schwartz functions, then
\formula[eq:vector:generalised]{
 \biggl\lV \biggl( \sum_{l = 1}^L (M_*^{(\alpha)} f_l)^q \biggr)^{1 / q} \biggr\rV_p & \le C_{d, \alpha, p, q} \biggl\lV \biggl( \sum_{l = 1}^L \lv f_l \rv^q \biggr)^{1 / q} \biggr\rV_p .
}
\end{theorem}

\begin{theorem}[Fefferman--Stein inequality for derivatives of spherical means]
\label{thm:vector:stein}
Suppose that
\formula[eq:range:vector]{
 & k = 0, 1, 2, \ldots\, \qquad d \ge 2 k + 3 , \qquad p, q \in \biggl(\frac{d}{d - k - 1} , \frac{d}{k + 1} \biggr) .
}
If\/ $f_1, f_2, \ldots, f_L$ are Schwartz functions, then
\formula[eq:vector:stein]{
 \biggl\lV \biggl( \sum_{l = 1}^L (A_*^{(k)} f_l)^q \biggr)^{1 / q} \biggr\rV_p & \le C_{k, p, q} \biggl\lV \biggl( \sum_{l = 1}^L \lv f_l \rv^q \biggr)^{1 / q} \biggr\rV_p .
}
\end{theorem}

In the remaining part of this section, we prove the above theorems. We follow the argument of~\cite{dk}, with some modifications. For example, we simplify the proof of Lemma~\ref{lem:high:crude} by estimating the Fourier transform instead of applying the Funk--Hecke formula.

For $p, q \in (1, \infty)$ and Schwartz functions $f_1, f_2, \ldots, f_L$, we use a short-hand notation
\formula{
 \lV (f_1, f_2, \ldots, f_L) \rV_{p, q} & = \biggl( \int_{\R^d} \biggl( \sum_{l = 1}^L \lv f_l(x) \rv^q \biggr)^{p / q} dx \biggr)^{1 / p}
}
for the $L^p(\ell^q)$ norm that appears in Fefferman--Stein estimates~\eqref{eq:vector:hardy:littlewood}, \eqref{eq:vector:generalised} and~\eqref{eq:vector:stein}.


\subsection{General estimates}

Let $\psi$ be a real-valued function on $[0, \infty)$ such that $\psi(\lv s \rv)$ is an even Schwartz function on $\R$. For $r > 0$ and $x \in \R^d$, let
\formula{
 \psi_r(x) & = r^{-d} \psi(r^{-1} \lv x \rv) .
}
Observe that $\fourier \psi_r(\xi) = \fourier \psi_1(r \xi)$, and $\fourier \psi_1$ is a radial function with profile given by the Hankel transform of $\psi$. We denote this transform by $\fourier_d \psi$; thus,
\formula[eq:hankel]{
 \fourier \psi_r(\xi) & = \fourier_d \psi(r \lv \xi \rv) = (2 \pi)^{d / 2} \int_0^\infty s^{d - 1} \psi(s) (r s \lv \xi \rv)^{1 - d / 2} J_{d / 2 - 1}(r s \lv \xi \rv) ds .
}
If $f$ is a Schwartz function, we define the convolution operators
\formula{
 \Psi_r f(x) & = f * \psi_r(x) .
}
Clearly,
\formula{
 \fourier \Psi_r f(\xi) & = \fourier \psi_r(\xi) \fourier f(\xi) = \fourier_d \psi(r \lv \xi \rv) \fourier f(\xi) .
}
We denote by $\Psi_*$ the corresponding maximal function,
\formula{
 \Psi_* f(x) & = \sup_{r \in (0, \infty)} \lv \Psi_r f(x) \rv ,
}
and by $g_\psi$ the associated square function,
\formula{
 g_\psi(f)(x) & = \biggl( \int_0^\infty \frac{\lv \Psi_r f(x) \rv^2}{r} \, dr \biggr)^{1 / 2} .
}
This notation is used without further comments in the remaining part of the paper.

The following three results are standard. The first one is obtained by controlling $\Psi_* f$ in terms of the Hardy--Littlewood maximal function $M_{\smash{*}}^{(1)} f$. The next one is a simple estimate of the square function $g_\psi(f)$. The last one turns it into a maximal inequality. We include the proof for the reader's convenience.

\begin{lemma}[see Proposition~2.1 in~\cite{dk}]
\label{lem:vector:low}
If\/ $p, q \in (1, \infty)$ and $f_1, f_2, \ldots, f_L$ are Schwartz functions, then
\formula{
 \lV (\Psi_* f_1, \Psi_* f_2, \ldots, \Psi_* f_L) \rV_{p, q} & \le C_{d, p, q} \lV \psi_1^* \rV_1^{\phantom{*}} \lV (f_1, f_2, \ldots, f_L) \rV_{p, q} ,
}
where $\psi_{\smash{1}}^*(x) = \sup \{ \lv \psi_1(y) \rv : \lv y \rv \ge \lv x \rv \}$ is the radially nonincreasing majorant of $\psi_1$.
\end{lemma}

\begin{proof}
By Corollary~2.1.12 in~\cite{grafakos}, for every Schwartz function $f$ and $x \in \R^d$, we have
\formula{
 \Psi_* f(x) & \le \lV \psi_1^* \rV_1^{\phantom{*}} M_*^{(1)} f(x) .
}
To obtain the desired result, we apply the above estimate to $f_1, f_2, \ldots, f_L$ and combine it with the classical Fefferman--Stein inequality~\eqref{eq:vector:hardy:littlewood}.
\end{proof}

\begin{lemma}[Proposition~2.2 in~\cite{dk}]
\label{lem:vector:square}
Let $0 < \varrho_0 < \varrho_1$, and suppose that
\formula{
 \fourier_d \psi(t) & = 0 \qquad \text{if\/ $t \le \varrho_0$ or\/ $t \ge \varrho_1$.}
}
If\/ $f$ is a Schwartz function, then
\formula{
 \lV g_\psi(f) \rV_2 & \le \lV \fourier_d \psi \rV_\infty \biggl( \log \frac{\varrho_1}{\varrho_0} \biggr)^{1 / 2} \lV f \rV_2 .
}
\end{lemma}

\begin{proof}
By a repeated application of Fubini's and Plancherel's theorems,
\formula{
 \lV g_\psi(f) \rV_2^2 & = \int_0^\infty \frac{\lV \Psi_r f \rV_2^2}{r} \, dr \\ 
 & = \frac{1}{(2 \pi)^d} \int_0^\infty \frac{\lV \fourier \Psi_r f \rV_2^2}{r} \, dr \displaybreak[0] \\
 & = \frac{1}{(2 \pi)^d} \int_{\R^d} \biggl( \int_0^\infty \frac{\lv \fourier_d \psi(r \xi) \rv^2}{r} \, dr \biggr) \lv \fourier f(\xi) \rv^2 d\xi \displaybreak[0] \\
 & \le \frac{1}{(2 \pi)^d} \lV \fourier_d \psi \rV_\infty^2 \int_{\R^d} \bigl( \log(\lv \xi \rv^{-1} \varrho_1) - \log(\lv \xi \rv^{-1} \varrho_0) \bigr) \lv \fourier f(\xi) \rv^2 d\xi \\
 & = \frac{1}{(2 \pi)^d} \lV \fourier_d \psi \rV_\infty^2 \log \biggl( \frac{\varrho_1}{\varrho_0} \biggr) \lV \fourier f \rV_2^2 = \lV \fourier_d \psi \rV_\infty^2 \log \biggl( \frac{\varrho_1}{\varrho_0} \biggr) \lV f \rV_2^2,
}
as desired.
\end{proof}

\begin{lemma}[see Proposition~3.2 in~\cite{dk}]
\label{lem:vector:high}
Let $0 < \varrho_0 < \varrho_1$, and suppose that
\formula{
 \fourier_d \psi(t) & = 0 \qquad \text{if\/ $t \le \varrho_0$ or\/ $t \ge \varrho_1$.}
}
If\/ $f$ is a Schwartz function, then
\formula{
 \lV \Psi_* f \rV_2 & \le \sqrt{2} \lV \fourier_d \psi \rV_\infty^{1 / 2} \lV \fourier_d \tilde \psi \rV_\infty^{1 / 2} \biggl( \log \frac{\varrho_1}{\varrho_0} \biggr)^{1 / 2} \lV f \rV_2 ,
}
where $\tilde \psi(s) = -s \psi'(s) - d \psi(s)$.
\end{lemma}

\begin{proof}
We denote $\tilde \psi_r(x) = r^{-d} \tilde \psi(r^{-1} \lv x \rv)$. Observe that $\fourier \tilde \psi_r(\xi) = \fourier_d \tilde \psi(r \lv \xi \rv)$, and
\formula{
 r \tfrac{d}{d r} \psi_r(x) & = r \tfrac{d}{d r} \bigl( r^{-d} \psi(r^{-1} \lv x \rv) \bigr) \\
 & = -d r^{-d} \psi(r^{-1} \lv x \rv) - r^{-d - 1} \lv x \rv \psi'(r^{-1} \lv x \rv) \\
 & = r^{-d} \tilde \psi(r^{-1} \lv x \rv) = \tilde \psi_r(x) .
}
Hence, by $\fourier \psi_r(\xi) = \fourier \psi_1(r \xi)$ and the dominated convergence theorem,
\formula[eq:psi:tilde]{
 \fourier \tilde \psi_r(\xi) = r \xi \cdot \nabla \fourier \psi_1(r \xi) .
}
Recall that $\fourier \psi_r(\xi) = \fourier_d \psi(r \lv \xi \rv)$. By a repeated use of Plancherel's, Fubini's, and the dominated convergence theorems, one easily proves that for every $x \in \R^d$, $\Psi_r f(x)$ is differentiable with respect to $r \in (0, \infty)$,
\formula*[eq:psi:derivative]{
 \fourier \bigl( r \tfrac{d}{d r} \Psi_r f \bigr) (\xi) & = r \tfrac{d}{d r} \bigl( \fourier \Psi_1 f(r \xi) \bigr) \\
 & = r \tfrac{d}{dr} \bigl( \fourier \psi_1(r \xi) \bigr) \fourier f(\xi) \\
 & = r \xi \cdot \nabla \fourier \psi_1(r \xi) \fourier f(\xi) \\
 & = \fourier \smash{\tilde \psi_1}(r \xi) \fourier f(\xi) = \fourier_d \smash{\tilde \psi}(r \lv \xi \rv) \fourier f(\xi) ,
}
and
\formula{
 \lim_{r \to \infty} \Psi_r f(x) & = 0 .
}
It follows that for $s \in (0, \infty)$,
\formula{
 \lv \Psi_s f(x) \rv^2 & = -\int_s^\infty \tfrac{d}{d r} \bigl( \lv \Psi_r f(x) \rv^2 \bigr) dr \le 2 \int_s^\infty \lv \Psi_r f(x) \rv \lv \tfrac{d}{d r} \Psi_r f(x) \rv dr ,
}
and hence
\formula{
 \lv \Psi_* f(x) \rv^2 & \le 2 \int_0^\infty \lv \Psi_r f(x) \rv \lv \tfrac{d}{d r} \Psi_r f(x) \rv dr .
}
By the Cauchy--Schwarz inequality,
\formula{
 \lV \Psi_* f \rV_2^2 & \le 2 \int_{\R^d} \int_0^\infty \lv \Psi_r f(x) \rv \lv \tfrac{d}{d r} \Psi_r f(x) \rv dr dx \\
 & \le 2 \biggl( \int_{\R^d} \int_0^\infty r^{-1} \lv \Psi_r f(x) \rv^2 dr dx \biggr)^{1 / 2} \biggl( \int_{\R^d} \int_0^\infty r^{-1} \lv r \tfrac{d}{d r} \Psi_r f(x) \rv^2 dr dx \biggr)^{1 / 2} .
}
By~\eqref{eq:psi:derivative}, the right-hand side is equal to $2 \lV g_\psi(f) \rV_2 \lV g_{\smash{\tilde \psi}}(f) \rV_2$, and so
\formula{
 \lV \Psi_* f \rV_2^2 & \le 2 \lV g_\psi(f) \rV_2 \lV g_{\tilde \psi}(f) \rV_2 .
}
The result follows now from Lemma~\ref{lem:vector:square}.
\end{proof}

Lemma~\ref{lem:vector:high} immediately leads to the following vector-valued estimate.

\begin{corollary}
\label{cor:vector:l2}
Let $0 < \varrho_0 < \varrho_1$, and suppose that
\formula{
 \fourier_d \psi(t) & = 0 \qquad \text{if\/ $t \le \varrho_0$ or\/ $t \ge \varrho_1$.}
}
If\/ $f_1, f_2, \ldots, f_L$ are Schwartz functions, then
\formula{
 \lV (\Psi_* f_1, \Psi_* f_2, \ldots, \Psi_* f_L) \rV_{2, 2} & \le C_d \lV \fourier_d \psi \rV_\infty^{1 / 2} \lV \fourier_d \tilde \psi \rV_\infty^{1 / 2} \biggl( \log \frac{\varrho_1}{\varrho_0} \biggr)^{1 / 2} \lV (f_1, f_2, \ldots, f_L) \rV_{2, 2} ,
}
where $\tilde \psi(s) = -s \psi'(s) - d \psi(s)$.
\end{corollary}


\subsection{Stein's generalised spherical averages}

In this section we prove Theorem~\ref{thm:vector:generalised}. Recall that Stein's generalised spherical means $M_{\smash{r}}^{(\alpha)} f(x)$ are given by~\eqref{eq:generalised} if $\re \alpha > 0$. When $\re \alpha \le 0$, $M_{\smash{r}}^{(\alpha)} f$ is the convolution of $f$ with an appropriate Schwartz distribution supported in the closed unit ball. In either case, the operators $M_{\smash{r}}^{(\alpha)}$ are Fourier multipliers, in the sense that
\formula{
 \fourier M_r^{(\alpha)} f(\xi) & = \hat M^{(\alpha)}(r \xi) \fourier f(\xi) ,
}
and the symbol $\hat M^{(\alpha)}$ is given by
\formula{
 \hat M^{(\alpha)}(\xi) & = 2^{d / 2 + \alpha - 1} \pi^{d / 2} \lv \xi \rv^{1 - d / 2 - \alpha} J_{d / 2 + \alpha - 1}(\lv \xi \rv) ;
}
see Theorem~IV.4.15 on p.~171 in~\cite{sw} or Equation~(6) in~\cite{k}. It is well-known that for $\nu \in \C$ and $s > 0$,
\formula[eq:bessel:estimate]{
 \lv s^{-\nu} J_\nu(s) \rv & \le C_\nu (1 + s)^{-\re \nu - 1 / 2}
}
(see \nist{10.14}{4} and \nist{10.17}{3} in~\cite{nist}). Furthermore, $\tfrac{d}{ds} (s^{-\nu} J_\nu(s)) = -s^{-\nu} J_{\nu + 1}(s)$ (see \nist{10.6}{6} in~\cite{nist}), and hence
\formula{
 \lv \tfrac{d}{ds} (s^{-\nu} J_\nu(s)) \rv & \le C_\nu (1 + s)^{-\re \nu - 1 / 2} .
}
It follows that
\formula*[eq:symbol:estimate]{
 \lv \hat M^{(\alpha)}(\xi) \rv & \le C_{d, \alpha} (1 + \lv \xi \rv)^{1 / 2 - d / 2 - \re \alpha} , \\
 \lv \nabla \hat M^{(\alpha)}(\xi) \rv & \le C_{d, \alpha} (1 + \lv \xi \rv)^{1 / 2 - d / 2 - \re \alpha}
}
for every $d \ge 1$, $\alpha \in \C$ and $\xi \in \R^d$.

Following the idea of Bourgain from~\cite{bourgain}, as in~\cite{dk}, we decompose the averaging operator $M_{\smash{r}}^{(\alpha)}$ into countably many terms $M_{\smash{r}}^{(\alpha, n)}$, $n = 0, 1, \ldots$\,, which correspond to frequencies of the order $2^n r^{-1}$. More precisely, we fix a smooth radial function $\hat \ph$ on $\R^d$ such that $0 \le \hat \ph \le 1$, $\hat \ph(\xi) = 1$ when $\lv \xi \rv \le 1$, and $\hat \ph(\xi) = 0$ if $\lv \xi \rv \ge 2$. We define $M_{\smash{r}}^{(\alpha, n)}$ using the Fourier transform:
\formula{
 \fourier M_r^{(\alpha, 0)} f(\xi) & = \hat \ph(r \xi) \hat M^{(\alpha)}(r \xi) \fourier f(\xi)
}
is the low frequency term, and
\formula[eq:high:symbol]{
 \fourier M_r^{(\alpha, n)} f(\xi) & = \bigl( \hat \ph(2^{-n} r \xi) - \hat \ph(2^{-n + 1} r \xi) \bigr) \hat M^{(\alpha)}(r \xi) \fourier f(\xi)
}
for $n = 1, 2, \ldots$ are the high frequency terms. Clearly, for every $\xi \in \R^d$ and $r > 0$,
\formula{
 \sum_{n = 0}^\infty \fourier M_r^{(\alpha, n)} f(\xi) & = \fourier M_r^{(\alpha)} f(\xi) .
}
Since $\fourier M_{\smash{r}}^{(\alpha)} f$ is a Schwartz function, the convergence is dominated by the integrable function $\lv \fourier M_{\smash{r}}^{(\alpha)} f(\xi) \rv$. Thus, the Fourier inversion formula and the dominated convergence theorem imply that
\formula{
 \sum_{n = 0}^\infty M_r^{(\alpha, n)} f(x) & = M_r^{(\alpha)} f(x)
}
for every $x \in \R^d$.

We denote by $M_{\smash{*}}^{(\alpha, n)} f$ the associated maximal functions:
\formula{
 M_*^{(\alpha, n)} f(x) & = \sup_{r \in (0, \infty)} \lv M_r^{(\alpha, n)} f(x) \rv .
}
Of course,
\formula{
 M_*^{(\alpha)} f(x) & \le \sum_{n = 0}^\infty M_*^{(\alpha, n)} f(x)
}
for every $x \in \R^d$.

In order to prove Theorem~\ref{thm:vector:generalised}, we later interpolate between the following $L^2(\ell^2)$ bound and the $L^p(\ell^q)$ estimate given in the next statement.

\begin{lemma}[see Proposition~3.2 in~\cite{dk}]
\label{lem:high:sharp}
If\/ $n \ge 1$, $\alpha \in \C$, and $f_1, f_2, \ldots, f_L$ are Schwartz functions, then
\formula{
 \lV (M_*^{(\alpha, n)} f_1, M_*^{(\alpha, n)} f_2, \ldots, M_*^{(\alpha, n)} f_L) \rV_{2, 2} & \le \frac{C_d}{2^{(d / 2 + \re \alpha - 1) n}} \lV (f_1, f_2, \ldots, f_L) \rV_{2, 2} .
}
\end{lemma}

\begin{proof}
Fix $n = 1, 2, \ldots$\, By definition, $M_{\smash{r}}^{(\alpha, n)}$ is a Fourier multiplier given by~\eqref{eq:high:symbol}, and hence a convolution operator with kernel $\psi_r(x) = r^{-d} \psi(r^{-1} \lv  x \rv)$, where
\formula[eq:psi:shell]{
 \fourier_d \psi(\lv \xi \rv) & = \bigl( \hat \ph(2^{-n} \xi) - \hat \ph(2^{-n + 1} \xi) \bigr) \hat M^{(\alpha)}(\xi) .
}
Note that $\psi$ depends on $n$. Denote $\tilde \psi(s) = -s \psi'(s) - d \psi(s)$ and $\tilde \psi_r(x) = r^{-d} \tilde \psi(r^{-1} \lv x \rv)$. Recall that by~\eqref{eq:psi:tilde},
\formula{
 \fourier_d \tilde \psi(\lv \xi \rv) & = \fourier \tilde \psi_1(\xi) = \xi \cdot \nabla \fourier \psi_1(\xi) .
}
Combining this with~\eqref{eq:symbol:estimate} and~\eqref{eq:psi:shell}, we find that
\formula{
 \lv \fourier_d \psi(t) \rv & \le C_{d, \alpha} (1 + t)^{1 / 2 - d / 2 - \re \alpha} , \\
 \lv \fourier_d \tilde \psi(t) \rv & \le C_{d, \alpha} (1 + t)^{3 / 2 - d / 2 - \re \alpha} .
}
Furthermore, by the definition of $\hat \ph$, $\fourier_d \psi(t) = \fourier_d \tilde \psi(t) = 0$ if $t \le 2^{n - 1}$ or $t \ge 2^{n + 1}$. Thus,
\formula{
 \lV \fourier_d \psi \rV_\infty & \le C_{d, \alpha} 2^{-(d / 2 + \re \alpha - 1 / 2) n} , \\
 \lV \fourier_d \tilde \psi \rV_\infty & \le C_{d, \alpha} 2^{-(d / 2 + \re \alpha - 3 / 2) n} .
}
The desired estimate follows now directly from Corollary~\ref{cor:vector:l2}.
\end{proof}

\begin{lemma}[see Proposition~3.3 in~\cite{dk}]
\label{lem:high:crude}
If\/ $n \ge 1$, $\alpha \in \C$, $p, q \in (1, \infty)$, and $f_1, f_2, \ldots, f_L$ are Schwartz functions, then
\formula{
 \lV (M_*^{(\alpha, n)} f_1, M_*^{(\alpha, n)} f_2, \ldots, M_*^{(\alpha, n)} f_L) \rV_{p, q} & \le C_{d, \alpha, p, q} 2^{(1 - \re \alpha) n} \lV (f_1, f_2, \ldots, f_L) \rV_{p, q} .
}
\end{lemma}

\begin{proof}
Fix $n = 1, 2, \ldots$\, Let $\psi$ and $\psi_r$ be defined as in the proof of Lemma~\ref{lem:high:sharp}; in particular,
\formula[eq:psi:fourier]{
 \fourier_d \psi(\lv \xi \rv) & = \fourier \psi_1(\xi) = \bigl( \hat \ph(2^{-n} \xi) - \hat \ph(2^{-n + 1} \xi) \bigr) \hat M^{(\alpha)}(\xi) .
}
We claim that
\formula[eq:psi:estimate]{
 \lv \psi(s) \rv & \le C_{d, \alpha} 2^{(1 - \re \alpha) n} (1 + s)^{-d - 1}
}
for all $s \ge 0$. Once this is established, the desired result follows immediately from Lemma~\ref{lem:vector:low}.

Using~\eqref{eq:symbol:estimate} and the definition of $\hat \ph$, we find that
\formula[eq:hankel:estimate]{
 \lv \fourier \psi_1(\xi) \rv & \le C_{d, \alpha} 2^{-(d / 2 + \re \alpha - 1 / 2) n}
}
if $2^{n - 1} \le \lv \xi \rv \le 2^{n + 1}$, and $\fourier \psi_1(\xi) = 0$ otherwise. By~\eqref{eq:psi:fourier}, the function $\psi$ is the inverse Hankel transform of $\fourier_d \psi$, the radial profile of $\fourier \psi_1$. Thus, as in~\eqref{eq:hankel}, we have
\formula{
 \psi(s) & = (2 \pi)^{-d / 2} \int_0^\infty t^{d - 1} \fourier_d \psi(t) (s t)^{1 - d / 2} J_{d / 2 - 1}(s t) dt .
}
Combining this with~\eqref{eq:hankel:estimate}, we obtain
\formula{
 \lv \psi(s) \rv & \le C_{d, \alpha} 2^{-(d / 2 + \re \alpha - 1 / 2) n} \int_{2^{n - 1}}^{2^{n + 1}} t^{d - 1} \bigl\lv (s t)^{1 - d / 2} J_{d / 2 - 1}(s t) \bigr\rv dt .
}
Using the estimate~\eqref{eq:bessel:estimate} of the Bessel function, we obtain
\formula{
 \lv \psi(s) \rv & \le C_{d, \alpha} 2^{-(d / 2 + \re \alpha - 1 / 2) n} \int_{2^{n - 1}}^{2^{n + 1}} t^{d - 1} (s t)^{1 / 2 - d / 2} dt \\
 & \le C_{d, \alpha} 2^{(1 - \re \alpha) n} s^{1 / 2 - d / 2} .
}
This clearly implies~\eqref{eq:psi:estimate} when $\tfrac{1}{2} \le s \le \tfrac{3}{2}$.

Let $\ph$ be the inverse Fourier transform of $\hat \ph$. Denote
\formula{
 \eta(x) & = \ph(x) - 2^{-d} \ph(\tfrac{1}{2} x) ,
}
and let $\eta_r(x) = r^{-d} \eta(r^{-1} x)$. The function $\hat \ph(2^{-n} \xi) - \hat \ph(2^{-n + 1} \xi)$ is thus the Fourier transform of $\eta_{2^{-n}}$. By~\eqref{eq:psi:fourier}, it follows that
\formula{
 \psi_1(x) & = M_{\smash{1}}^{(\alpha)} \eta_{2^{-n}}(x) .
}
Stein's generalised mean $M_{\smash{1}}^{(\alpha)} \eta_r(x)$ is the convolution of $\eta_r$ with a Schwartz distribution, which is equal to zero on $\R^d \setminus \overline{\ball(0, 1)}$, equal to a smooth function on $\ball(0, 1)$, and singular on $\partial \ball(0, 1)$; see~\cite{stein}. Since $\fourier \eta$ is equal to $0$ in $\ball(0, \tfrac{1}{2})$, all moments of $\eta_r(x) = r^{-d} \eta(r^{-1} x)$ are equal to zero. A somewhat technical, but standard ``approximate identity'' argument shows that away from $\partial \ball(0, 1)$, $M_{\smash{1}}^{(\alpha)} \eta_r(x)$ converges uniformly and rapidly to zero as $r \to 0^+$. That is, for every $m \ge 0$ and $\eps > 0$, we have
\formula{
 \lv M_1^{(\alpha)} \eta_r(x) \rv & \le C_{d, \alpha, \eps, m} r^m (1 + \lv x \rv)^{-m},
}
uniformly with respect to $x$ such that $\lv x \rv \le 1 - \eps$ or $\lv x \rv \ge 1 + \eps$. We postpone the details to Lemma~\ref{lem:auxiliary} in Appendix~\ref{sec:convergence}. It follows that
\formula{
 \lv \psi_1(x) \rv & \le C_{d, \alpha, \eps, m} 2^{-m n} (1 + \lv x \rv)^{-m} ,
}
uniformly with respect to $x$ such that $\lv x \rv \le 1 - \eps$ or $\lv x \rv \ge 1 + \eps$. By choosing $\eps = \tfrac{1}{2}$ and $m$ such that $m \ge \re \alpha - 1$ and $m \ge d + 1$, we obtain~\eqref{eq:psi:estimate} for $s \le \tfrac{1}{2}$ or $s \ge \tfrac{3}{2}$. This completes the proof.
\end{proof}

\begin{lemma}[see the proof of Theorem~1.2 in~\cite{dk}]
\label{lem:high}
If\/ $n \ge 1$, $\alpha \in \C$, $p, q \in (1, \infty)$,
\formula[eq:high:condition]{
 \gamma & < d \min \biggl\{ \frac{1}{p} , 1 - \frac{1}{p}, \frac{1}{q} , 1 - \frac{1}{q} \biggr\} + \re \alpha - 1 ,
}
and $f_1, f_2, \ldots, f_L$ are Schwartz functions, then
\formula{
 \biggl\lV \biggl( \sum_{l = 1}^L (M_*^{(\alpha, n)} f_l)^q \biggr)^{1 / q} \biggr\rV_p & \le C_{d, \alpha, p, q, \gamma} 2^{-\gamma n} \biggl\lV \biggl( \sum_{l = 1}^L \lv f_l \rv^q \biggr)^{1 / q} \biggr\rV_p .
}
\end{lemma}

\begin{proof}
Suppose that $p_0, q_0 \in (1, \infty)$ and $\thet \in (0, 1)$ are chosen so that if $p_1 = q_1 = 2$, then
\formula{
 \frac{1}{p} & = \frac{1 - \thet}{p_0} + \frac{\thet}{p_1} , \qquad \frac{1}{q} = \frac{1 - \thet}{q_0} + \frac{\thet}{q_1} .
}
Note that given $\thet \in (0, 1)$, such a choice of $p_0$ and $q_0$ is possible as long as
\formula{
 \thet & < 2 \min \biggl\{ \frac{1}{p} , 1 - \frac{1}{p}, \frac{1}{q} , 1 - \frac{1}{q} \biggr\} .
}
Consider the linear operator
\formula{
 T^{(\alpha, n)} F & = G , 
}
acting on vectors $F = (f_1, f_2, \ldots, f_L)$ of Schwartz functions $f_1, f_2, \ldots, f_L$, and taking values $G = (g_1, g_2, \ldots, g_L)$, where
\formula{
 g_l(x, r) & = M_r^{(\alpha, n)} f_l(x) .
}
Define the $L^p(\ell^q)$ norm of $F$ as usual,
\formula{
 \lV F \rV_{p, q} & = \biggl( \int_{\R^d} \biggl(\sum_{l = 1}^L \lv f_l(x) \rv^q \biggr)^{p / q} dx \biggr)^{1 / p} ,
}
and let the $L^p(\ell^q(L^\infty))$ norm of $G$ be
\formula{
 \lV G \rV_{p, q, \infty} & = \biggl( \int_{\R^d} \biggl(\sum_{l = 1}^L \biggl( \sup_{r \in (0, \infty)} \lv g_l(x, r) \rv \biggr)^q \biggr)^{p / q} dx \biggr)^{1 / p} .
}
By Lemma~\ref{lem:high:sharp}, we know that $T^{(\alpha, n)}$ is a bounded operator from $L^2(\ell^2)$ to $L^2(\ell^2(L^\infty))$, with norm at most $C_d 2^{-(d / 2 + \re \alpha - 1) n}$. On the other hand, by Lemma~\ref{lem:high:crude}, the norm of $T^{(\alpha, n)}$ as an operator from $L^{p_0}(\ell^{q_0})$ to $L^{p_0}(\ell^{q_0}(L^\infty))$ does not exceed $C_{\smash{d, \alpha, p_0, q_0}} 2^{(1 - \re \alpha) n}$. By an appropriate variant of the Riesz--Thorin interpolation theorem, $T^{(\alpha, n)}$ is a bounded operator from $L^p(\ell^q)$ to $L^p(\ell^q(L^\infty))$, with norm at most
\formula{
 C_{d, \alpha, p_0, q_0} 2^{(1 - \thet) (1 - \re \alpha) n - \thet (d / 2 + \re \alpha - 1) n} & = C_{d, \alpha, p_0, q_0} 2^{-(\thet d / 2 + \re \alpha - 1) n} .
}
By choosing $\thet$ such that $\tfrac{\thet d}{2} + \re \alpha - 1 \ge \gamma$, we obtain the desired result.
\end{proof}

We are ready to prove~\eqref{eq:vector:generalised}.

\begin{proof}[Proof of Theorem~\ref{thm:vector:generalised}]
By assumption, there is $\gamma > 0$ satisfying~\eqref{eq:high:condition}. Lemma~\ref{lem:high} implies that
\formula{
 \lV (M_*^{(\alpha, n)} f_1, M_*^{(\alpha, n)} f_2, \ldots, M_*^{(\alpha, n)} f_L) \rV_{p, q} & \le C_{d, \alpha, p, q} 2^{-\gamma n} \lV (f_1, f_2, \ldots, f_L) \rV_{p, q}
}
for every $n \ge 1$. Furthermore, $M_{\smash{r}}^{(\alpha, 0)}$ is a Fourier multiplier with radial symbol $\hat \ph(r \xi) \hat M^{(\alpha)}(r \xi)$. By Lemma~\ref{lem:vector:low} applied to the radial profile $\psi$ of the symbol of $M_{\smash{1}}^{(\alpha, 0)}$, we obtain
\formula{
 \lV (M_*^{(\alpha, 0)} f_1, M_*^{(\alpha, 0)} f_2, \ldots, M_*^{(\alpha, 0)} f_L) \rV_{p, q} & \le C_{d, \alpha, p, q} \lV (f_1, f_2, \ldots, f_L) \rV_{p, q} .
}
The monotonicity and the convexity of the norm lead to
\formula{
 \lV (M_*^{(\alpha)} f_1, M_*^{(\alpha)} f_2, \ldots, M_*^{(\alpha)} f_L) \rV_{p, q} & \le \sum_{n = 0}^\infty \lV (M_*^{(\alpha, n)} f_1, M_*^{(\alpha, n)} f_2, \ldots, M_*^{(\alpha, n)} f_L) \rV_{p, q} \\
 & \le C_{d, \alpha, p, q} \lV (f_1, f_2, \ldots, f_L) \rV_{p, q} ,
}
as desired.
\end{proof}


\subsection{Derivatives of spherical averages}

The proof of~\eqref{eq:vector:stein} is now quite simple.

\begin{proof}[Proof of Theorem~\ref{thm:vector:stein}]
In Section~2.7 in~\cite{k}, it is argued that if the Fefferman--Stein inequality~\eqref{eq:vector:stein} holds in dimension $d$ with some constant, then it automatically holds in every higher dimension, with the same constant. Thus, it is sufficient to prove~\eqref{eq:vector:stein} with a constant $C_{d, k, p, q}$ that may depend on the dimension.

By Equation~(9) in~\cite{k}, for every $k \ge 0$ we have
\formula{
 A_r^{(k)} f(x) & = \frac{2^k \Gamma(\tfrac{d}{2})}{\pi^{d / 2}} \biggl( M_r^{(-k)} f(x) + \sum_{j = 0}^{k - 1} C_{d, k, j} M_r^{(-j)} f(x) \biggr) .
}
By an application of Theorem~\ref{thm:vector:generalised} with $\alpha = -j$, $j = 0, 1, \ldots, k$, we obtain the desired Fefferman--Stein inequality~\eqref{eq:vector:stein}, with a dimension-dependent constant $C_{d, k, p, q}$. As remarked above, this completes the proof.
\end{proof}

%
%

\appendix
\section{Rapid convergence result}
\label{sec:convergence}

In this section we prove a convergence result needed in the proof of Lemma~\ref{lem:high:crude}. We begin with an auxiliary lemma.

\begin{lemma}
\label{lem:auxiliary:singular}
Suppose that $\Lambda$ is a compactly supported Schwartz distribution supported in a compact set $K$. Let $\eta$ be a Schwartz function, and define $\eta_r(x) = r^{-d} \eta(r^{-1} x)$. Then, for every $\eps > 0$, $m \ge 0$, $r \in (0, 1]$ and $x \in \R^d$ such that $\dist(x, K) \ge \eps$, we have
\formula{
 \lv \Lambda * \eta_r(x) \rv & \le C_{\Lambda, \eta, K, \eps, m} \frac{r^m}{(1 + \lv x \rv)^m} .
}
\end{lemma}

\begin{proof}
Let $k$ be the order of the distribution $\Lambda$. Fix $x \in \R^d \setminus K$ and denote $\delta = \dist(x, K)$. Since $\Lambda = 0$ in the ball $\ball(x, \delta)$, we have
\formula{
 \lv \Lambda * \psi(x) \rv & \le C_\Lambda \sup \{ \lv \partial^\alpha \psi(y) \rv : \lv \alpha \rv \le k , \lv y \rv \ge \delta \} ;
}
here $\alpha$ is a multi-index, $\partial^\alpha$ is the corresponding partial derivative, and $y \in \R^d$. Substituting $\psi = \eta_r$ with $r \in (0, 1]$, we obtain
\formula{
 \lv \Lambda * \eta_r(x) \rv & \le C_\Lambda \sup \{ r^{-\lv \alpha \rv} \lv \partial^\alpha \eta(r^{-1} y) \rv : \lv \alpha \rv \le k , \lv y \rv \ge \delta \} \\
 & \le C_\Lambda r^{-k} \sup \{ \lv \partial^\alpha \eta(z) \rv : \lv \alpha \rv \le k , \lv z \rv \ge r^{-1} \delta \} .
}
Since $\eta$ is a Schwartz function, it follows that
\formula{
 \lv \Lambda * \eta_r(x) \rv & \le C_{\Lambda, \eta} r^{-k} (1 + r^{-1} \delta)^{-k - 2 m} .
}
If $\delta \ge \eps$, we obtain
\formula{
 \lv \Lambda * \eta_r(x) \rv & \le C_{\Lambda, \eta} r^{-k} (r^{-1} \eps)^{-k - m} (1 + \delta)^{-m} = C_{\Lambda, \eta, \eps, m} r^m (1 + \delta)^{-m} .
}
It remains to observe that $1 + \lv x \rv \le C_{K, \eps} (1 + \delta)$.
\end{proof}

\begin{lemma}
\label{lem:auxiliary}
Suppose that $\Lambda$ is a compactly supported Schwartz distribution, which is equal to a smooth function $f$ in the complement of a compact set $K$. Let $\eta$ be a Schwartz function such that all moments of $\eta$ are equal to zero, and define $\eta_r(x) = r^{-d} \eta(r^{-1} x)$. Then, for every $\eps > 0$, $m \ge 0$, $r \in (0, 1]$ and $x \in \R^d$ such that $\dist(x, K) \ge \eps$, we have
\formula{
 \lv \Lambda * \eta_r(x) \rv & \le C_{\Lambda, \eta, \eps, m} \frac{r^m}{(1 + \lv x \rv)^m} .
}
\end{lemma}

\begin{proof}
Fix $\eps > 0$ and a Schwartz function $\ph$ such that $\ph(x) = 1$ if $\dist(x, K) \le \tfrac{1}{3} \eps$ and $\ph(x) = 0$ if $\dist(x, K) \ge \tfrac{2}{3} \eps$. Denote
\formula{
 \Lambda_0 & = \ph \Lambda , \qquad \Lambda_1 = (1 - \ph) \Lambda = (1 - \ph) f .
}
Clearly, $\Lambda = \Lambda_0 + \Lambda_1$.

By Lemma~\ref{lem:auxiliary:singular} applied to the distribution $\Lambda_0$, the compact set $\{x \in \R^d : \dist(x, K) \le \tfrac{2}{3} \eps\}$, and $\eps$ replaced by $\tfrac{1}{3} \eps$, we have
\formula[eq:auxiliary:rough]{
 \lv \Lambda_0 * \eta_r(x) \rv & \le C_{\Lambda, \ph, \eta, K, \eps, m} \frac{r^m}{(1 + \lv x \rv)^m}
}
for every $m \ge 0$, $r \in (0, 1]$ and $x \in \R^d$ such that $\dist(x, K) \ge \eps$.

Similarly, applying Lemma~\ref{lem:auxiliary:singular} to the distribution $\Lambda_1$, the compact set $L$ equal to the support of $\Lambda$, and $\eps$ replaced by $1$, we obtain
\formula[eq:auxiliary:smooth]{
 \lv \Lambda_1 * \eta_r(x) \rv & \le C_{\Lambda, \ph, \eta, m} \frac{r^m}{(1 + \lv x \rv)^m}
}
for every $m \ge 0$, $r \in (0, 1]$ and $x \in \R^d$ such that $\dist(x, L) \ge 1$.

Furthermore, $\Lambda_1 = (1 - \ph) f$ is a smooth, compactly supported function. Using the Fourier inversion formula and the fact that $\fourier \eta_r(\xi) = \fourier \eta(r \xi)$, we find that
\formula{
 \lv \Lambda_1 * \eta_r(x) \rv & \le \frac{1}{(2 \pi)^d} \int_{\R^d} \lv \fourier \Lambda_1(\xi) \rv \lv \fourier \eta(r \xi) \rv d\xi .
}
Since all moments of $\eta$ are equal to zero, we have
\formula{
 \lv \fourier \eta(\xi) \rv & \le C_{\eta, m} \lv \xi \rv^m ,
}
and hence
\formula{
 \lv \Lambda_1 * \eta_r(x) \rv & \le \frac{C_{\eta, m} r^m}{(2 \pi)^d} \int_{\R^d} \lv \fourier \Lambda_1(\xi) \rv \lv \xi \rv^m d\xi = C_{\Lambda, \ph, \eta, m} r^m
}
for every $m \ge 0$, $r \in (0, 1]$ and $x \in \R^d$. It follows that~\eqref{eq:auxiliary:smooth} also holds when $\dist(x, L) < 1$. The desired estimate follows now by combining~\eqref{eq:auxiliary:rough} and~\eqref{eq:auxiliary:smooth}.
\end{proof}

%
%

\section*{}

\subsection*{Acknowledgements}

We thank Błażej Wróbel for inspiring discussions about the subject of the paper.

%
%

%
%

\end{document}